\documentclass{amsart}
\usepackage{dsfont, amssymb, mathtools, IEEEtrantools, subcaption, dutchcal}
\usepackage{tabularray} 
\usepackage{calc,changepage}
\usepackage{csquotes}
\usepackage{enumitem}
\usepackage{booktabs}
\usepackage{diagbox}

\usepackage{xcolor}
\definecolor{wb}{RGB}{51,153,255}
\usepackage{tikz}
\usetikzlibrary{calc}
\usepackage{tikz-cd}

\usepackage[parfill]{parskip}
\numberwithin{equation}{subsection}
\usepackage{amsrefs}

\newcommand{\defeq}{\vcentcolon=}
\newcommand{\eqdef}{=\vcentcolon}

\makeatletter
\def\moverlay{\mathpalette\mov@rlay}
\def\mov@rlay#1#2{\leavevmode\vtop{%
   \baselineskip\z@skip \lineskiplimit-\maxdimen
   \ialign{\hfil$\m@th#1##$\hfil\cr#2\crcr}}}
\newcommand{\charfusion}[3][\mathord]{
    #1{\ifx#1\mathop\vphantom{#2}\fi
        \mathpalette\mov@rlay{#2\cr#3}
      }
    \ifx#1\mathop\expandafter\displaylimits\fi}
\makeatother

\newcommand{\bigcupdot}{\charfusion[\mathop]{\bigcup}{\cdot}}

\DeclareFontFamily{U}{mathb}{\hyphenchar\font45}
\DeclareFontShape{U}{mathb}{m}{n}{
      <5> <6> <7> <8> <9> <10> gen * mathb
      <10.95> mathb10 <12> <14.4> <17.28> <20.74> <24.88> mathb12
      }{}
\DeclareSymbolFont{mathb}{U}{mathb}{m}{n}
\DeclareMathSymbol{\precneq}{3}{mathb}{"AC}
\DeclareMathSymbol{\varprec}{3}{mathb}{"A0}

\newtheoremstyle{definitions}
 	{\topsep}
	{\topsep}
	{}
	{}
	{\bfseries}
	{:}
	{.5em}
	{}
  
\newtheoremstyle{lemmata}
	{\topsep}
	{\topsep}
	{\itshape} 
	{}
	{\bfseries}
	{:}
	{.5em}
	{}

\theoremstyle{lemmata}
\newtheorem{Theorem}[subsection]{Theorem}
\newtheorem{Lemma}[subsection]{Lemma}
\newtheorem{Corollary}[subsection]{Corollary}
\newtheorem{Proposition}[subsection]{Proposition}

\theoremstyle{definitions}

\newtheorem{Remark}[subsection]{Remark}
\newtheorem{Remarks}[subsection]{Remarks}
\newtheorem*{Remarks-nn}{Remarks}

\newtheorem{Example}[subsection]{Example}

\newtheorem{Problem}[subsection]{Problem}
\newtheorem{Property}[subsection]{Property}

\DeclareMathOperator{\GL}{GL}
\DeclareMathOperator{\SL}{SL}

\DeclareMathOperator{\aut}{aut}

\DeclareMathOperator{\Characteristic}{char}
\DeclareMathOperator{\Gr}{Gr}
\DeclareMathOperator{\id}{id}
\DeclareMathOperator{\res}{res}

\usepackage{hyperref}
\allowdisplaybreaks

\title[Hecke operators and growth of expansion coefficients]{Modular forms for \(\GL(r, \mathds{F}_{q}[T])\): Hecke operators and growth of expansion coefficients}
\author{Ernst-Ulrich Gekeler \\ Universität des Saarlandes \\ \lowercase{\href{mailto:gekeler@math.uni-sb.de}{gekeler@math.uni-sb.de}}}
\address{FR 6.1 Mathematik, Universität des Saarlandes, Postfach 15 11 50 D-66041 Saarbrücken}
\email{gekeler@math.uni-sb.de}
\date{\today}

\begin{document}

\DefTblrTemplate{firsthead,middlehead,lasthead,capcont,contfoot-text}{default}{} 

\begin{abstract}
	We determine the action of the Hecke operators \(T_{\mathfrak{p},i}\) on the coefficient forms \(g_{1}, \dots, g_{r-1}, g_{r} = \Delta\), and
	\(h\), which together generate the ring of modular forms for \(\GL(r, \mathds{F}_{q}[T])\). All these are eigenforms with powers of \(\pi\)
	as eigenvalues, where \(\pi\) is the monic generator of the prime ideal \(\mathfrak{p}\) of \(\mathds{F}_{q}[T]\). We further describe the 
	growth of the \(t\)-expansion coefficients of the discriminant function \(\Delta\). It is such that the product expansion of \(\Delta\) as well
	as the \(t\)-expansion of each modular form converges on the natural fundamental domain for \(\GL(r, \mathds{F}_{q}[T])\).
\end{abstract}

\maketitle

\setcounter{section}{-1}

\section{Introduction}

Drinfeld modular forms in rank two were introduced by David Goss in the seventies \cite{Goss80}, \cite{Goss80-2} and, a bit later, by the present author \cite{Gekeler80}. 
The theory, which has striking analogies with the theory of classical elliptic modular forms (see e.g. \cite{Gekeler85}, \cite{Gekeler86}), was continued in the
eighties and nineties (e.g., \cite{Gekeler86-2}, \cite{Gekeler88}, \cite{GekelerReversat96}), where basic facts about the relationship with Drinfeld modular curves, congruence properties,
and the Shimura-Taniyama-Weil uniformization of certain elliptic curves over function fields were established. 

In the meantime, due to the efforts of many mathematicians, Drinfeld modular forms of rank two have turned into a flourishing field of number
theoretic research. The extension of the theory to higher rank \(r > 2\) is relatively new, and is still in its beginning. Dirk Basson \cite{Basson14}, \cite{Basson17},
Basson-Breuer \cite{BassonBreuer17} and Gekeler \cite{Gekeler19} (but see also \cite{Gekeler92}) dealt with special aspects of higher rank modular forms, while Basson-Breuer-Pink in \cite{BassonBreuerPink24}
gave a systematic foundation of the theory, including a comparison between \enquote{algebraic} and \enquote{analytic} Drinfeld modular forms. 
(It should be mentioned that \cite{BassonBreuerPink24} is the essentially unchanged conjunction of three preprints that appeared already in 2017 
in the ArXiv.)

Independently, the present author started his systematic investigations in the field with \cite{Gekeler17}, from which a series of meanwhile seven articles
with title \enquote{On higher rank Drinfeld modular forms I-VII} emanated, e.g. \cite{Gekeler22-2} and \cite{Gekeler25}.

Hecke operators in the higher rank theory were first defined in \cite{BassonBreuerPink24} via double cosets of arithmetic groups. The authors show that they act
\enquote{essentially trivially} on Eisenstein series (loc. cit. Theorem 14.11). It is the aim of the present article to investigate arithmetic
properties of the Hecke action in the technically most simple case of modular forms for the full modular group \(\Gamma = \GL(r, \mathds{F}_{q}[T])\).

For each prime ideal \(\mathfrak{p}\) of \(A = \mathds{F}_{q}[T]\) with monic generator \(\pi\), there exist the \(r+1\) Hecke operators
\(T_{\mathfrak{p}, i}\) (\(0 \leq i \leq r\)), where \(r \geq 2\) is a fixed natural number. Regarded as a correspondence on the set 
\(\mathcal{L}^{r}\) of \(A\)-lattices \(\Lambda\) of rank \(r\) in the field \(C_{\infty}\) of \enquote{complex numbers} (\(C_{\infty}\) is the
completed algebraic closure of \(K_{\infty} = \mathds{F}_{q}((T^{-1}))\)\,), \(T_{\mathfrak{p},i}\) associates with \(\Lambda\) the finite set of
super-lattices \(\widetilde{\Lambda}\) of \(\Lambda\) such that \(\widetilde{\Lambda}/\Lambda\) is isomorphic with \((\mathds{F}_{\mathfrak{p}})^{i}\),
where \(\mathds{F}_{\mathfrak{p}} = A/\mathfrak{p}\). Hence \(T_{\mathfrak{p},0}(\Lambda) = \{ \Lambda \}\), 
\(T_{\mathfrak{p},r}(\Lambda) = \{ \mathfrak{p}^{-1}\Lambda\}\) are trivial, while \(T_{\mathfrak{p},i}\) with \(1 \leq i < r\) is meaningful.

The structure of the abstract local Hecke algebra \(\mathbf{H}_{\mathfrak{p}} = \mathds{Z}[T_{\mathfrak{p},i} \mid 1 \leq i \leq r]\) has been
determined in \cite{Shimura71} Theorem 3.20\footnote{In fact, Shimura worked over the discrete valuation ring \(\mathds{Z}_{(p)} =\) localization of 
\(\mathds{Z}\) over a prime \(p\), but the generalization to the dvr \(A_{(\mathfrak{p})} =\) localization of \(A\) at \(\mathfrak{p}\) is obvious.};
it turns out that the \(T_{\mathfrak{p},i}\) are algebraically independent, and so \(\mathbf{H}_{\mathfrak{p}}\) is a polynomial ring in the 
\(T_{\mathfrak{p},i}\). Letting them act on spaces of modular forms in the usual fashion, operators from \(\mathbf{H}_{\mathfrak{p}}\) and 
\(\mathbf{H}_{\mathfrak{q}}\) (\(\mathfrak{p} \neq \mathfrak{q}\) prime ideals) commute; so we may restrict to studying them separately. 

The algebra \(\mathbf{M}^{r}\) of modular forms for \(\Gamma\) is generated by the \(r\) algebraically independent coefficient forms
\(g_{1}, \dots, g_{r-1}, g_{r} = \Delta\), and the \((q-1)\)-th root \(h\) of \(\Delta\). All these are eigenforms for all the \(T_{\mathfrak{p},i}\),
with powers of \(\pi\) as eigenvalues, see Theorem \ref{Theorem.Basic-coefficient-forms-are-Eigenforms-of-Hecke-operators}. In the case of
rank \(r = 2\), this phenomenon (only the innocent eigenvalues \enquote{powers of \(\pi\)} for a modular form \(f\)) allows for a new structure, the
existence of an \(A\)-expansion for \(f\) is in the style of Lopez \cite{LopezBartolome10} and Petrov \cite{Petrov13}. We wonder which 
consequences this might have in the present framework of \(r > 2\). Here, no \(A\)-expansions beyond those of Eisenstein series 
(see \eqref{Eq.Ek-omega-as-Ek-omega-prime-minus-Goss-polynomials}) are known.

Apart from the preparations and the proof of Theorem \ref{Theorem.Basic-coefficient-forms-are-Eigenforms-of-Hecke-operators}, we also study the 
behaviour of the coefficients \(a_{n}(f)\) of the \(t\)-expansion for \(f = \Delta\), and for some related functions. As the \(a_{n}(f)\) for
\(r > 2\) are not constant but itself (weak) modular forms of lower rank \(r-1\), we restrict our consideration to arguments \(\boldsymbol{\omega}\)
of \(f\) which actually lie in the fundamental domain \(\mathbf{F} = \mathbf{F}^{r}\) for \(\Gamma\); so the \(a_{n}(f)\) are functions on
\(\mathbf{F}^{r-1}\), the fundamental domain in rank \(r-1\). The growth of \(a_{n}(f)\) for \(n \to \infty\) in the sub-region 
\(\mathbf{F}_{\mathbf{k}} \subset \mathbf{F}\) is described in Theorem \ref{Theorem.Fundamental-index}. In particular, the product expansions for 
\(f \in \{h, \Delta\}\) converge uniformly on all of \(\mathbf{F}\). This result (convergence of \(t\)-expansions on all of \(\mathbf{F}\))
holds in fact for all modular forms by Corollary \ref{Corollary.Convergence-of-all-t-expansions-of-modular-forms-on-fundamental-domain}. So we can
safely perform on \(\mathbf{F}\) the usual analytic operations with modular forms.

About the plan of the paper: Section \ref{Section.The-starting-point} introduces the basic concepts, definitions, and terminology. In Section
\ref{Section.Known-results-about-series-expansions-of-modular-forms}, some more technical objects (Goss polynomials, the reciprocal division 
polynomials \(S_{n}(X)\)\,) are presented, which are needed to write down the (known) expansions \ref{Subsection.t-expansions-of-Eisenstein-series} and
\ref{Theorem.Product-expression-discriminant-Delta-omega} of Eisenstein series and the discriminant. In Section \ref{Section.Hecke-operators},
the Hecke operators \(T_{\mathfrak{p},i}\) are studied as correspondences on the set \(\mathcal{L}^{r}\) of rank-\(r\) lattices in \(C_{\infty}\),
and thus as endomorphisms of the group \(\mathds{Z}[\mathcal{L}^{r}]\) of divisors on \(\mathcal{L}^{r}\). The definition is extended to the set
\(\mathcal{L}^{r,\pm}\) of oriented lattices, to allow Hecke actions on modular forms with non-trivial type. We prefer this approach to that via
double cosets as in \cite{BassonBreuerPink24}, as it is more conceptual and transparent.\footnote{But we must note here that the natural extension of this approach
to modular forms with level (say \(\mathfrak{n}\), an ideal of \(A\)) necessarily requires an adelic framework, as level-\(\mathfrak{n}\) structues
on lattices involve different connected components of the moduli scheme \(\mathcal{M}^{r}(\mathfrak{n}) \times C_{\infty}\).}
In order to cover the \(t\)-expansions, we introduce in Section \ref{Section.The-effect-of-Hecke-operators-on-t-expansions} the notion of splitting
of the \(r\)-lattice \(\Lambda\). It allows an induction step from rank \(r\) to rank \(r-1\), the distinction of type-\(1\) and type-\(2\) parts
of the Hecke correspondence and the simple but powerful Proposition \ref{Proposition.T-p-i-on-lattice} that describes the restriction of Hecke
operators to the boundary. The principal technical result of this section is Proposition 
\ref{Proposition.t-expansion-of-f-in-M-k-l-to-power-series-expansion-of-T-p-i-f} on the effect of Hecke operators to \(t\)-expansions. It looks
hair-rising on a first view, but admits at least the induction step in the proof of Theorem 
\ref{Theorem.Basic-coefficient-forms-are-Eigenforms-of-Hecke-operators}. We further add a similar result on the effect to the terms of possible
\(A\)-expansions, with a slight hope that such expansions could exist.

In Section \ref{Section.The-Hecke-action-on-the-basic-modular-forms} we first show the crucial congruence property \ref{Property.Power-series-in-t}
for the expansions of the elementary modular forms of type \(0\), which plays a key role in the proof of Theorem 
\ref{Theorem.Basic-coefficient-forms-are-Eigenforms-of-Hecke-operators}. We then transfer the trivial principle that \enquote{forms that generate
a Hecke-stable one-dimensional space must be eigenforms} from lower-rank algebras \(\mathbf{M}^{r'}\) of modular forms to \(\mathbf{M}^{r}\)
and get that \(g_{1}, \dots, g_{r-1}, \Delta, h\) (along with some other forms) are eigenforms \textbf{a priori}. The eigenvalues are then
determined in Theorem \ref{Theorem.Basic-coefficient-forms-are-Eigenforms-of-Hecke-operators}.

Finally, Section \ref{Section.Growth-coefficients} deals with growth properties of \(t\)-expansion coefficients \(a_{n}(f)\), notably for
\(f = \Delta\). Such expansions have by definition a positive radius of convergence, locally in the boundary point where the expansion takes place.
Our results imply that we have global convergence on the full fundamental domain \(\mathbf{F}\).

\textbf{Notation.} \(r\) is a fixed natural number, the \textbf{rank}. Usually \(r \geq 2\); in some cases \(r = 1\) is tacitly allowed for
induction purposes.

\(\mathds{F} = \mathds{F}_{q}\), the finite field with \(q\) elements, of characteristic \(p\). The quantities \(r\) and \(q\) are mostly omitted
from notation.

\(A = \mathds{F}[T]\), the polynomial ring in an indeterminate \(T\), \(K = \mathds{F}(T)\) its quotient field. The completion \(\mathds{F}((T^{-1}))\)
of \(K\) at infinity is denoted by \(K_{\infty}\), its ring of integers by \(O_{\infty}\), its completed algebraic closure by \(C_{\infty}\).

On \(C_{\infty}\) we have: its absolute value \(\lvert \mathop{.} \rvert\), normalized by \(\lvert T \rvert = q\); \(O_{C_{\infty}}=\{x \in C_{\infty} \mid \lvert x \rvert \leq 1\}\), \(\mathfrak{m}_{C_{\infty}} = \{ x \in C_{\infty} \mid \lvert x \rvert < 1\}\), with 
\(O_{C_{\infty}}/\mathfrak{m}_{C_{\infty}} \overset{\cong}{\rightarrow} \overline{\mathds{F}}\), the algebraic closure of \(\mathds{F}\). 

\begin{center}
	\begin{longtblr}{Q[r,t]Q[c,b]Q[l,t,0.9\textwidth]}
		\(\Psi\)							& \(=\)	& \(\Psi^{r} = \{ \boldsymbol{\omega} = (\omega_{1}, \ldots, \omega_{r}) \in C_{\infty}^{r} \mid \text{the \(\omega_{i}\) are \(K_{\infty}\)-linearly independent}\}\) \\
											&		& and \\
		\(\Omega\)							& \(=\)	& \(\Omega^{r} = \Psi^{r}/C_{\infty}^{*} = \{ \boldsymbol{\omega} \in \mathds{P}^{r-1}(C_{\infty}) \mid \boldsymbol{\omega} \text{ represented by some element of } \Psi^{r}\}\), \\
											&		& often identified with \(\{ \boldsymbol{\omega} = (\omega_{1}, \ldots, \omega_{r-1},1) \} \subset \Psi^{r}\). \\
		\(\Gamma\)							& \(=\)	& \(\GL(r,A)\), with fundamental domain \(\mathbf{F} \subset \Omega\); \\
		\(\mathbf{F}_{\mathbf{k}}\)			& 		& distinguished subspace of \(\mathbf{F}\) associated with \(\mathbf{k} = (k_{1}, \ldots, k_{r}) \in \mathds{N}_{0}^{r}\), \(k_{1} \geq k_{2} \geq \cdots \geq k_{r}=0\); \\
		\(\mathbf{M}\)						& \(=\)	& \(\bigoplus_{k,\ell} M_{k,\ell} = C_{\infty}[g_{1}, \dots, g_{r-1}, g_{r} = \Delta, h]\) the algebra of modular forms for \(\Gamma\), with subalgebra \(\mathbf{M}_{0} = \bigoplus_{k} M_{k,0}\); \\
		\(t\)								& \(=\) & \(t(\boldsymbol{\omega})\) uniformizer of \(\Omega\) \enquote{along the boundary}; \\
		\(\mathds{P}\)						& \(=\)	& \(\mathds{P}^{r-1}_{q-1,\dots,q^{r}-1} = \bigcupdot_{1 \leq i \leq r} \mathcal{M}^{i}\), the compactified moduli scheme, \(\mathcal{M}^{r}(C_{\infty}) = \Gamma \backslash \Omega\); \\
		\(\mathcal{L}\)						& \(=\)	& \(\mathcal{L}^{r}\) (resp. \(\mathcal{L}^{\pm} = \mathcal{L}^{r,\pm}\)) the set of \(A\)-lattices of rank \(r\) (resp. of oriented \(A\)-lattices of rank \(r\)); \\
		\(\phi\)							& \(=\)	& \(\phi^{\boldsymbol{\omega}}\) the generic Drinfeld \(A\)-module of rank \(r\), \(\boldsymbol{\omega} \in \Omega\), with operator polynomial \(\phi_{T}(X) = TX + g_{1}(\boldsymbol{\omega})X^{q} + \dots + g_{r-1}(\boldsymbol{\omega})X^{q^{r-1}} + \Delta(\boldsymbol{\omega})X^{q^{r}}\);\\
		\(\mathds{N}\)						& \(=\)	& \(\{1,2,3,\dots \}\); \\
		\(\mathds{N}_{0}\)					& \(=\)	& \(\{0,1,2,3,\dots \}\); \\
		\(\sideset{}{^{\prime}}\sum\limits_{i \in I} x_{i}\) && is the sum \(\sum_{0 \neq i \in I} x_{i}\), and similarly \(\sideset{}{^{\prime}}\prod x_{i} = \prod_{i \neq 0} x_{i}\); \\
		\(\lvert S \rvert\)					& 		& is the cardinality of the set \(S\).						
	\end{longtblr}
\end{center}

\section{The starting point}\label{Section.The-starting-point}

(See \cite{Gekeler17} and \cite{Gekeler22} for more details)

\subsection{The relevant spaces}

We let 
\begin{align*}
	\Psi 	&= \Psi^{r} = \{ \boldsymbol{\omega} = (\omega_{1}, \ldots, \omega_{r}) \in C_{\infty}^{r} \mid \text{the \(\omega_{i}\) are \(K_{\infty}\)-linearly independent} \} 
	\intertext{and} 
	\Omega	&= \Omega^{r} = \Psi^{r}/C_{\infty}^{*}
\end{align*}
be the \textbf{Drinfeld symmetric spaces}, provided with their natural structures as rigid-analytic spaces over \(K_{\infty}\). They are acted upon
by the group \(\GL(r, K_{\infty})\) and therefore by the modular group \(\Gamma = \GL(r,A)\). Usually we normalize the last coordinate \(\omega_{r}\)
of \(\boldsymbol{\omega} \in \Omega\) to \(\omega_{r} = 1\) and write \(\boldsymbol{\omega} = (\omega_{1}, \dots, \omega_{r-1},1)\); then the
action of \(\gamma = (\gamma_{ij}) \in \GL(r, K_{\infty})\) is by fractional linear transformations
\begin{align}\stepcounter{subsubsection}%
	\gamma(\boldsymbol{\omega})	&= \gamma \boldsymbol{\omega} = \boldsymbol{\omega}'=(\omega_{1}', \dots, \omega_{r-1}',1) \text{ with} \\
			\omega_{i}'			&= \aut(\gamma, \boldsymbol{\omega})^{-1} \sum_{1 \leq j \leq r} \gamma_{i,j}\omega_{j} \quad (1 \leq i < r) \quad \text{and} \nonumber \\
	\aut(\gamma, \boldsymbol{\omega})	&= \sum_{1 \leq j \leq r} \gamma_{r,j} \omega_{j}.	\nonumber 
\end{align}
\(\Omega\) is related with the \textbf{Bruhat-Tits- building} \(\mathcal{BT} = \mathcal{BT}^{r}\) of \(\GL(r,K_{\infty})\) through the surjective 
building map \(\lambda \colon \Omega \to \mathcal{BT}(\mathds{Q})\). Recall that \(\mathcal{BT}\) is a contractible simplicial complex of dimension
\(r-1\) on which \(\GL(r, K_{\infty})\) acts, transitively on simplices of dimension \(i\) (\(0 \leq i < r\)). The set \(\mathcal{BT}(\mathds{R})\)
of points of the realization of \(\mathcal{BT}\) corresponds to the set of similarity classes of norms on the \(K_{\infty}\)-vector space
\(K_{\infty}^{r}\). A norm \(\lVert \mathop{.} \rVert\) represents an integral point (i.e., a point of \(\mathcal{BT}(\mathds{Z}) =\) set of vertices
of \(\mathcal{BT}\)) if and only if \(\lVert \mathop{.} \rVert\) is similar to \(\lVert \mathop{.} \rVert_{L}\), the norm with unit ball \(L\), where
\(L\) is an \(O_{\infty}\)-lattice in \(K_{\infty}^{r}\).

Now \(\lambda\) is the map \(\boldsymbol{\omega} \mapsto \lVert \mathop{.} \rVert_{\boldsymbol{\omega}}\), where 
\(\lVert \mathop{.} \rVert_{\boldsymbol{\omega}}\) is the norm
\begin{equation}\stepcounter{subsubsection}%
	\lVert \mathbf{x} \rVert_{\boldsymbol{\omega}} \defeq \lvert \mathbf{x} \boldsymbol{\omega} \rvert = \Big\lvert \sum_{1 \leq i \leq r} x_{i}\omega_{i} \Big\rvert 
\end{equation}
on \(K_{\infty}^{r}\). As the absolute value \(\lvert \mathop{.} \rvert\) on \(C_{\infty}\) has values in \(q^{\mathds{Q}} \cup \{ 0 \}\), the
class of the norm \(\lVert \mathop{.} \rVert_{\boldsymbol{\omega}}\) belongs in fact to \(\mathcal{BT}(\mathds{Q})\). 

For given \(\mathbf{x} \in \mathcal{BT}(\mathds{Q})\), we let
\begin{equation}\stepcounter{subsubsection}%
	\Omega_{\mathbf{x}} \defeq \{ \boldsymbol{\omega} \in \Omega \mid \lambda(\boldsymbol{\omega}) = \mathbf{x} \}
\end{equation}
be its fiber in \(\Omega\). Then \(\Omega_{\mathbf{x}}\) is an open affinoid subspace of \(\Omega\), the geometry of which depends on the simplex
\(\langle \mathbf{x} \rangle\) spanned by \(\mathbf{x}\), and is described in \cite{Gekeler22} Theorem 2.4. For a holomorphic function \(f\) on \(\Omega\),
\begin{equation}\stepcounter{subsubsection}%
	\lVert f \rVert_{\mathbf{x}} = \max_{\mathbf{x} \in \Omega_{\mathbf{x}}} \lvert f(\boldsymbol{\omega}) \rvert
\end{equation}
denotes its \textbf{spectral norm} on \(\Omega_{\mathbf{x}}\). Basic facts about \(\lambda\) are: 
\subsubsection{}\label{Subsubsection.Invertible-function-on-open-affinoid-subspace}%
if \(f\) is invertible on \(\Omega_{\mathbf{x}}\) then \(\lvert f \rvert\) is constant on \(\Omega_{\mathbf{x}}\);
\subsubsection{}\label{Subsubsection.Invertible-function-on-pre-image-of-closed-simplex}%
if \(f\) is invertible on the pre-image \(\lambda^{-1}(\sigma)\) of a closed simplex \(\sigma\) of \(\mathcal{BT}\), then
\[
	\log f \defeq \log_{q} \lVert f \rVert_{\mathbf{x}}
\]
interpolates linearly on \(\sigma(\mathds{Q})\), see \cite{Gekeler22} Theorem 2.4 and 2.6.

\subsection{The fundamental domain}\label{Subsection.The-fundamental-domain}

A finite subset \(S\) of \(C_{\infty}\) is \textbf{orthogonal} if for each family \((a_{s})_{s \in S}\) of coefficients in \(K_{\infty}\) the
rule
\begin{equation}
	\Big\lvert \sum_{s \in S} a_{s} s \Big\rvert = \max_{s \in S} \lvert a_{s} s \rvert 
\end{equation}
holds. 

An \(A\)-\textbf{lattice} in \(C_{\infty}\) is a finitely generated (hence free of certain rank) \(A\)-submodule \(\Lambda\) that is 
\textbf{discrete} in the sense that it has finite intersection with each ball in \(C_{\infty}\). With each 
\(\boldsymbol{\omega} = (\omega_{1}, \dots, \omega_{r-1},1) \in \Omega\), we associate the lattice
\begin{equation}
	\Lambda_{\boldsymbol{\omega}} = \sum_{1 \leq i \leq r} A \omega_{i}
\end{equation}
of rank \(r\). A \textbf{successive minimal basis} (SMB) of the lattice \(\Lambda\) is an \(A\)-basis \(\{ \lambda_{1}, \dots, \lambda_{r}\}\) which 
is orthogonal and satisfies \(\lvert \lambda_{1} \rvert \leq \lvert \lambda_{2} \rvert \leq \cdots \leq \lvert \lambda_{r} \rvert\). Such 
an SMB always exists and is obtained by choosing \(\lambda_{i} \in \Lambda \smallsetminus \sum_{1 \leq j < i} A\lambda_{j}\) such that 
\(\lvert \lambda_{i}\rvert\) is minimal under this condition, see \cite{Gekeler19} Section 3. Further, the sequence 
\(\lvert \lambda_{1} \rvert, \dots, \lvert \lambda_{r} \rvert\) is an invariant of \(\Lambda\), i.e., uniquely determined. As a consequence, each
\(\boldsymbol{\omega} \in \Omega\) may be transformed by \(\Gamma = \GL(r,A)\) to some \(\boldsymbol{\omega}'\) in
\begin{equation}
	\mathbf{F} \defeq \{ \boldsymbol{\omega} \in \Omega \mid \text{the \(\omega_{i}\) are orthogonal and } \lvert \omega_{1} \rvert \geq \cdots \geq \lvert \omega_{r} \rvert = 1 \};
\end{equation}
then \(\{\omega_{r} = 1, \omega_{r-1}, \dots, \omega_{1}\}\) is an SMB of \(\Lambda_{\boldsymbol{\omega}}\) (note the reverse order!) Hence each
\(\boldsymbol{\omega} \in \Omega\) is \(\Gamma\)-equivalent to at least one and at most finitely many elements of \(\mathbf{F}\).
This is why we call \(\mathbf{F}\) the \textbf{fundamental domain} for \(\Gamma\). It is an open analytic subspace of \(\Omega\) and the
pre-image \(\lambda^{-1}(\mathcal{W})\) under the building map, where \(\mathcal{W}\) is the full subcomplex of \(\mathcal{BT}\) with set of vertices
\begin{equation}\label{Eq.Set-of-vertices-W-IZ}
	\mathcal{W}(\mathds{Z}) = \{ [L_{\mathbf{k}}] \mid \mathbf{k} = (k_{1}, \dots, k_{r}) \in \mathds{N}_{0}^{r}, k_{1} \geq k_{2} \geq \cdots \geq k_{r} = 0\}.
\end{equation}
Here \(L_{\mathbf{k}}\) is the \(O_{\infty}\)-lattice \(O_{\infty} \pi^{k_{1}} \oplus \cdots \oplus O_{\infty}\pi^{k_{r}}\) in \(K_{\infty}^{r}\)
with similarity class \([L_{\mathbf{k}}]\), and \(\pi\) is the uniformizer \(T^{-1}\) of \(K_{\infty}\). (Beware of the sign error in \cite{Gekeler17} 2.2!)

We call \(\mathbf{k} \in \mathds{N}_{0}^{r}\) as in \eqref{Eq.Set-of-vertices-W-IZ} a \textbf{fundamental index}. For such \(\mathbf{k}\), we let
\begin{equation}\label{Eq.Distinguished-subdomain-of-F-defined-by-k}
	\mathbf{F}_{\mathbf{k}} = \lambda^{-1}([L_{\mathbf{k}}]) = \{ \boldsymbol{\omega} \in \Omega \mid \text{the \(\omega_{i}\) orthogonal and } \log_{q} \lvert \omega_{i} \rvert = k_{i} \} \subset \mathbf{F}	
\end{equation}
be the distinguished subdomain of \(\mathbf{F}\) defined by \(\mathbf{k}\).

\subsection{Modular forms}

A holomorphic function \(f\) on \(\Omega\) is a \textbf{Drinfeld modular form} for \(\Gamma\) of \textbf{weight} \(k \in \mathds{N}_{0}\) and
\textbf{type} \(\ell \in \mathds{Z}/(q-1)\) (briefly: of type \((k, \ell)\)) if 
\begin{align}
	&f(\gamma \boldsymbol{\omega}) = \aut(\gamma, \boldsymbol{\omega})^{k} (\det \gamma)^{-\ell} f(\boldsymbol{\omega}) && (\gamma \in \Gamma, \boldsymbol{\omega} \in \Omega) \tag{1.3.1)(i} \label{Eq.Characterization-Drinfeld-modular-form-(i)}\\
	&f \text{ is bounded on } \mathbf{F}. \tag{1.3.1)(ii} \label{Eq.Characterization-Drinfeld-modular-form-(ii)}
\end{align}\stepcounter{equation}%
(Condition (ii) is one of several equivalent ways how the boundary condition for \(f\) can be expressed, see \cite{Gekeler22-2}, Theorems 7.4 and 7.9.) Assuming
only condition (i), we call \(f\) a \textbf{weak modular form} of type \((k,\ell)\). We let \(M_{k,\ell}\) be the \(C_{\infty}\)-vector space
of modular forms of type \((k,\ell)\) and 
\begin{equation}
	\mathbf{M} = \mathbf{M}^{r} = \bigoplus_{\substack{k \in \mathds{N}_{0} \\ \ell \in \mathds{Z}/(q-1)}} M_{k,\ell}
\end{equation}
be the doubly graded algebra of modular forms with subalgebra \(\mathbf{M}_{0} = \bigoplus_{k \in \mathds{N}_{0}} M_{k,0}\).

\begin{Remarks-nn}
	\begin{enumerate}[wide]
		\item[(i)] As function spaces on \(\Omega\), the various \(M_{k,\ell}\) are linearly independent, so their sum is in fact direct.
		\item[(ii)] The definition of modular forms may easily be extended to congruence subgroups \(\Gamma'\) of \(\Gamma\), by requiring
		condition (i) for \(\gamma \in \Gamma'\) only and \(\lvert f_{[\gamma]_{k,\ell}} \rvert\) bounded on \(\mathbf{F}\) for a set of
		representatives of the finite set \(\Gamma' \backslash \Gamma\), with the transform
		\[
			f_{[\gamma]_{k,\ell}}(\boldsymbol{\omega}) \defeq \aut(\gamma, \boldsymbol{\omega})^{{-}k} (\det \gamma)^{\ell} f(\gamma \boldsymbol{\omega}).
		\]	
		\item[(iii)] If \(0 \neq f \in M_{k,\ell}\) then \(k = r\ell \pmod{q-1}\).  
	\end{enumerate}
\end{Remarks-nn}\stepcounter{equation}%

The most simple example of Drinfeld modular form is the \textbf{Eisenstein series}, defined by
\begin{equation}
	E_{k}(\boldsymbol{\omega}) \defeq \sideset{}{^{\prime}} \sum_{\mathbf{a} \in A^{r}} (\mathbf{a} \boldsymbol{\omega})^{{-}k},
\end{equation}
which gives an element of \(M_{k,0}\), nonzero if \(k \equiv 0 \pmod{q-1}\). Here 
\(\mathbf{a} \boldsymbol{\omega} = \sum_{1 \leq i \leq r} a_{i}\omega_{i}\), and the primed sum \(\sideset{}{^{\prime}}\sum\) is, as always, 
the sum over the non-zero elements \(\mathbf{a}\) of the index set.

\subsection{Drinfeld modules}\label{Subsection.Drinfeld-modules}

We let \(\mathcal{L} = \mathcal{L}^{r}\) be the set of \(A\)-lattices of rank \(r\) in \(C_{\infty}\). With \(\Lambda \in \mathcal{L}\) we 
can associate
\begin{itemize}
	\item the \textbf{exponential function} \(e^{\Lambda} \colon C_{\infty} \to C_{\infty}\);
	\item the \textbf{Drinfeld module} \(\phi^{\Lambda} = C_{\infty}/\Lambda\) over \(C_{\infty}\).	
\end{itemize}
The exponential function is defined by the product
\begin{equation}
	e^{\Lambda}(z) = z \sideset{}{^{\prime}} \prod_{\lambda \in \Lambda} (1 - z/\lambda);
\end{equation}
it is everywhere convergent and may be written as an entire power series
\begin{equation}\label{Eq.Powerseries-expansion-exponential-function}
	e^{\Lambda}(z) = \sum_{i \geq 0} \alpha_{i}z^{q^{i}}
\end{equation}
with \(\alpha_{0} = 1\) and \(c^{q^{i}} \alpha_{i} \to 0\) for each \(c \in C_{\infty}\). The Drinfeld module \(\phi^{\Lambda}\) is an exotic structure
of \(A\)-module on \(C_{\infty}\) characterized by a commutative diagram with exact rows 
\begin{equation}
	\begin{tikzcd}
		0 \ar[r] 	& \Lambda \ar[r] \ar[d, "\times a"] & C_{\infty} \ar[r, "e^{\Lambda}"] \ar[d, "\times a"]	& C_{\infty} \ar[r] \ar[d, "\phi_{a}^{\Lambda}"]	& 0 \\
		0 \ar[r]	& \Lambda \ar[r]						& C_{\infty} \ar[r, "e^{\Lambda}"]						& C_{\infty} \ar[r]										& 0	
	\end{tikzcd}
\end{equation}
for each \(a\) in \(A\). The maps \(\phi_{a}^{\Lambda}\) are uniquely determined by \(\phi_{T}^{\Lambda}\), which is an \(\mathds{F}\)-linear
polynomial
\begin{equation}
	\phi_{T}^{\Lambda}(X) = TX + g_{1}X^{q} + \cdots + g_{r}X^{q^{r}} \qquad (g_{1}, \dots, g_{r} \in C_{\infty}).
\end{equation}
The coefficient \(g_{r}\) is always non-zero, and is called the \textbf{discriminant} \(\Delta\) of \(\Phi^{\Lambda}\). For \(a \in A\) of
degree \(d\), \(\phi_{a}^{\Lambda}\) is a polynomial of shape \(\phi_{a}^{\Lambda}(X) = \sum_{0 \leq i \leq rd} \, \prescript{}{a} \ell_{i} X^{q^{i}}\),
with \(\prescript{}{a}\ell_{0} = a\), \(\prescript{}{a}\ell_{rd} \neq 0\). Conversely, given \(g_{1}, \dots, g_{r} \in C_{\infty}\) with
\(\Delta = g_{r} \neq 0\), there exists a unique \(r\)-lattice \(\Lambda \in \mathcal{L}\) such that \((g_{1}, \dots, g_{r})\) are the coefficients
of \(\phi_{T}^{\Lambda}\).

\subsection{Modular forms associated with Drinfeld modules}\label{Subsection.Modular-forms-associated-with-Drinfeld-modules}
Consider the \(C_{\infty}\)-valued function on \(\Omega\) that maps \(\boldsymbol{\omega}\) to the coefficient \(g_{i}\) of the Drinfeld module
\(\phi^{\boldsymbol{\omega}} \defeq \phi^{\Lambda_{\boldsymbol{\omega}}}\), where \(1 \leq i \leq r\). This gives a function 
\(\boldsymbol{\omega} \mapsto g_{i}(\boldsymbol{\omega})\) that may be verified to be an element of \(M_{q^{i}-1,0}\); i.e., it is holomorphic
and satisfies the conditions \eqref{Eq.Characterization-Drinfeld-modular-form-(i)}, \eqref{Eq.Characterization-Drinfeld-modular-form-(ii)}.
Similarly, for \(0 \neq a \in A\) and \(1 \leq i \leq r \cdot \deg a\), the coefficient \(\prescript{}{a}\ell_{i}\) defines a modular form
of weight \(q^{i}-1\) and type 0, a \textbf{coefficient form}. The \(g_{i} = \prescript{}{T}\ell_{i}\) are called the \textbf{basic coefficient forms}.
Finally, \(\boldsymbol{\omega} \mapsto \alpha_{i}(\boldsymbol{\omega}) =\) the coefficient \(\alpha_{i}\) of \(e^{\Lambda_{\boldsymbol{\omega}}}(z)\)
(see \eqref{Eq.Powerseries-expansion-exponential-function}) defines a modular form \(\alpha_{i} \in M_{q^{i}-1,0}\). It is called the
\textbf{para-Eisenstein series} of weight \(q^{i}-1\), as it shares some properties with the special Eisenstein series \(E_{q^{i}-1}\) of weight 
\(q^{i}-1\).

Describing modular forms with non-zero type is less trivial. It is known (\cite{Gekeler17} Theorem 3.8) that there exists a \((q-1)\)-th root \(h\) of the
discriminent function \(\Delta\) as a holomorphic function on \(\Omega\). There are several \enquote{natural} normalizations of \(h\); we
use the one that satisfies
\begin{equation}\stepcounter{subsubsection}%
	h^{q-1}(\boldsymbol{\omega}) = ({-}1)^{r-1} \Delta(\boldsymbol{\omega}). \label{Eq.Natural-normalization-of-h}
\end{equation}
It turns out that \(h\) is a modular form of weight \((q^{r}-1)/(q-1)\) and type \(1\).
\subsubsection{}\stepcounter{equation}%
The ring \(\mathbf{M}_{0} = \bigoplus_{k \geq 0} M_{k,0}\) of modular forms of type \(0\) may be written as the polynomial algebra in either 
\begin{enumerate}
	\item[(a)] \(g_{1}, \dots, g_{r} = \Delta\), or in
	\item[(b)] \(\alpha_{1}, \dots, \alpha_{r}\), or in
	\item[(c)] the special Eisenstein series \(E_{q-1}, E_{q^{2}-1}, \dots, E_{q^{r}-1}\). 
\end{enumerate}
The full algebra \(\mathbf{M} = \bigoplus_{k,\ell} M_{k,\ell}\) of all modular forms is generated over \(\mathbf{M}_{0}\) by \(h\) with the single 
relation \eqref{Eq.Natural-normalization-of-h}. For later use, we give the formulas that allow to determine the other forms from the special
Eisenstein series, see e.g. \cite{Gekeler86-2} p.13:
\begin{align}
	g_{k}	&= (T^{q^{k}} - T) E_{q^{k}-1} + \sum_{1 \leq j < k} E_{q^{k-j}-1} g_{j}^{q^{k-j}} && (1 \leq k \leq r) \label{Eq.Basic-coefficient-form-Eisenstein-series}
	\intertext{and}
	\sum_{\substack{i,j \geq 0 \\ i+j=k}} \alpha_{i} E_{q^{j}-1}^{q^{i}} &= \sum_{\substack{i,j \geq 0 \\ i+j=k}} E_{q^{i}-1}\alpha_{j}^{q^{i}} = 0, &&(k \in \mathds{N})	\label{Eq.Coefficients-of-exponential-function}
\end{align}
where \(\alpha_{1} = 1\) and \(E_{0} = {-}1\).

\subsection{The uniformizer}\label{Subsection.The-uniformizer}%
Write \(\boldsymbol{\omega} \in \Omega\) as \((\omega_{1}, \boldsymbol{\omega}')\) with 
\begin{align*}
	\boldsymbol{\omega}'=(\omega_{2}, \dots, \omega_{r-1}, \omega_{r}=1) \in \Omega^{r-1} && \text{(in case \(r=2\), \(\Omega^{r-1}=\Omega^{1}\defeq\{1\}\))}.
\end{align*}
We define the function \(t \colon \Omega \to C_{\infty}\) through
\begin{equation}
	t(\boldsymbol{\omega}) = (e^{\Lambda_{\boldsymbol{\omega}'}}(\omega_{1}))^{-1},
\end{equation}
where \(\Lambda_{\boldsymbol{\omega}'} = \sum_{2 \leq i \leq r} A\omega_{i}\) is the \((r-1)\)-lattice defined by \(\boldsymbol{\omega}'\).

It is holomorphic, vanishes nowhere on \(\Omega\), and satisfies
\begin{equation}
	t(\gamma \boldsymbol{\omega}) = \gamma_{1,1}^{-1} \aut(\gamma, \boldsymbol{\omega}) t(\boldsymbol{\omega})
\end{equation}
if \(\gamma \in \Gamma\) is such that \(\gamma_{2,1} = \gamma_{3,1} = \cdots = \gamma_{r,1} = 0\). In particular,
\begin{equation}\label{Eq.Invariance-of-t-under-transformation}
	t \text{ is invariant under } (\omega_{1}, \boldsymbol{\omega}') \mapsto (\omega_{1} + \lambda, \boldsymbol{\omega}'), \quad \lambda \in \Lambda_{\boldsymbol{\omega}'}.
\end{equation}
Each holomorphic function \(f\) on \(\Omega\) subject to \eqref{Eq.Invariance-of-t-under-transformation} has a Laurent expansion
\begin{equation}
	f(\boldsymbol{\omega}) = \sum_{n \in \mathds{Z}} a_{n}(\boldsymbol{\omega}') t^{n}(\boldsymbol{\omega})
\end{equation}
with respect to \(t\), where the \(a_{n}\) are holomorphic functions on \(\Omega^{r-1}\) (see \cite{Gekeler22-2} or \cite{BassonBreuerPink24}), and the sum converges for 
\(\lvert t \rvert\) small enough, locally in \(\boldsymbol{\omega}'\). If \(f\) is weakly modular (condition 
\eqref{Eq.Characterization-Drinfeld-modular-form-(i)}), the condition
\begin{equation}
	a_{n} \equiv 0 \text{ for } n < 0
\end{equation}
is equivalent with \eqref{Eq.Characterization-Drinfeld-modular-form-(ii)}, i.e., with \(f\) being modular. If even \(a_{n} \equiv 0\) for \(n \leq 0\),
\(f\) is called a \textbf{cusp form}. We let \(S_{k,\ell}\) be the space of cusp forms of type \((k,\ell)\). If \(f\) is modular of type \((k,\ell)\)
then
\begin{equation}
	a_{n} \equiv 0 \text{ unless } n \equiv k-r\ell \pmod{q-1}.
\end{equation}
It follows that \(\mathbf{S} \defeq \mathbf{S}^{r} = \bigoplus_{k,\ell} S_{k,\ell}\) is the ideal of \(\mathbf{M}\) generated by \(h\) and
\(\mathbf{S} \cap \mathbf{M}_{0}\) the ideal of \(\mathbf{M}_{0}\) generated by \(\Delta\).

\subsection{The moduli scheme}\label{Subsection.The-moduli-scheme}%
We let \(\mathds{P} = \mathds{P}_{q-1,\dots,q^{r}-1}^{r-1}\) be the weighted projective space \cite{Dolgachev82} of dimension \(r-1\) with system 
\((q-1, q^{2}-1, \dots, q^{r}-1)\) of weights. That is, its set of \(C_{\infty}\)-points is
\begin{equation}
	\mathds{P}(C_{\infty}) = \{ (x_{1} : \ldots : x_{r}) \mid (x_{1}, \ldots, x_{r}) \neq \mathbf{0} \},
\end{equation}
where the equivalence relation is \((x_{1}: \ldots : x_{r}) = (y_{1}: \ldots : y_{r})\) if and only if there exists \(0 \neq c \in C_{\infty}\)
such that \(y_{i} = c^{q^{i}-1}x_{i}\) for \(1 \leq i \leq r\). It follows from the \enquote{Weierstraß theory} of Drinfeld modules sketched
in \ref{Subsection.Drinfeld-modules} and \ref{Subsection.Modular-forms-associated-with-Drinfeld-modules} that
\begin{equation}
	(g_{1}(\boldsymbol{\omega}) : \ldots : g_{r}(\boldsymbol{\omega})) = (g_{1}(\boldsymbol{\eta}): \ldots : g_{r}(\boldsymbol{\eta}))
\end{equation}
in \(\mathds{P}(C_{\infty})\) if and only if the points \(\boldsymbol{\omega}, \boldsymbol{\eta} \in \Omega\) are conjugate under
\(\Gamma = \GL(r,A)\). As \(\Gamma\) operates discontinuously on \(\Omega\), there is a well-defined quotient analytic space 
\(\Gamma \backslash \Omega\), and the above shows that the map
\begin{equation}
	\begin{split}
		\Gamma \backslash \Omega 				&\longrightarrow \mathds{P}(C_{\infty}) \\
		\text{class of } \boldsymbol{\omega}		&\longmapsto (g_{1}(\boldsymbol{\omega}): \ldots : g_{r}(\boldsymbol{\omega}))
	\end{split}
\end{equation}
is well-defined and injective. Actually it is an open embedding of analytic spaces and identifies \(\Gamma \backslash \Omega\) with the subspace
\begin{equation}
	\mathcal{M}^{r}(C_{\infty}) \defeq \{ (g_{1}:\ldots:g_{r}) \in \mathds{P}(C_{\infty}) \mid g_{r} \neq 0 \}.
\end{equation}
Hence
\begin{equation}
	\overline{\mathcal{M}}^{r}(C_{\infty}) = \mathds{P}(C_{\infty}) = \bigcupdot_{1 \leq i \leq r} \mathcal{M}^{i}(C_{\infty}),
\end{equation}
where the similarly built
\begin{align*}
	\mathcal{M}^{i}(C_{\infty}) = \{ (g_{1}:\ldots:g_{i}) \mid g_{i} \neq 0 \} \subset \mathds{P}^{i-1}_{q-1,\ldots,q^{i}-1} &&(2 \leq i \leq r)	
\end{align*}
and \(\mathcal{M}^{1}(C_{\infty}) = \{1\}\) are naturally considered as the subspaces of \(\mathds{P}(C_{\infty})\) with last \(r-i\) coordinates
vanishing. As the notation indicates, \(\mathcal{M}^{i}(C_{\infty})\) is the set of \(C_{\infty}\)-points of the coarse moduli scheme 
\(\mathcal{M}^{i}\) for Drinfeld \(A\)-modules of rank \(i\), with natural compactification \(\overline{\mathcal{M}}^{i}\), which also equals the
Zariski closure of \(\mathcal{M}^{i}\) in \(\overline{\mathcal{M}}^{r}\). In the same vein, we write 
\begin{equation}
	\overline{\Omega} = \overline{\Omega}^{r} = \bigcupdot_{1 \leq i \leq r} \Omega^{i} \quad \text{and} \quad \overline{\mathbf{F}} = \overline{\mathbf{F}}^{r} = \bigcupdot_{1 \leq i \leq r} \mathbf{F}^{i}
\end{equation}
where \(\mathbf{F}^{i} = \{ (0,\ldots,0,\omega_{r-i+1}, \ldots, \omega_{r}=1) \mid \text{the \(\omega_{i}\) are orthogonal and } \lvert \omega_{r-i+1}\rvert \geq \cdots \geq \lvert \omega_{r} \rvert = 1 \}\).

\section{Known results about series expansions of modular forms}\label{Section.Known-results-about-series-expansions-of-modular-forms}
\subsection{} In order to state these, we need the following. Let \(\Lambda \subset C_{\infty}\) be a discrete (see 
\ref{Subsection.The-fundamental-domain}) \(\mathds{F}\)-subspace of \(C_{\infty}\). All the quantities that appear depend on \(\Lambda\),
which usually is omitted from notation. Let
\begin{equation}
	e(z) = e^{\Lambda}(z) = z \sideset{}{^{\prime}} \prod_{\lambda \in \Lambda} (1 - z/\lambda) = \sum_{k \in \mathds{N}_{0}} \alpha_{k}z^{q^{k}}
\end{equation}
be its exponential function, which generalizes \eqref{Eq.Powerseries-expansion-exponential-function}. It has a local inverse with respect to insertion
\begin{equation}
	\log^{\Lambda}(z) = \sum_{k \in \mathds{N}_{0}} \beta_{k}z^{q^{k}},
\end{equation}
that is, on a small ball around \(0\), \((e \circ \log)(z) = z = (\log \circ e)(z)\). We have
\begin{align}
	\beta_{k} = {-}\sideset{}{^{\prime}} \sum_{\lambda \in \Lambda} \lambda^{1-q^{k}} &&(k > 0) \text{ and } \beta_{0} = 1,
\end{align}
so \(\beta_{k}\) is minus the Eisenstein series \(E_{q^{k}-1}\) evaluated on \(\Lambda\) if \(\Lambda\) is an \(A\)-lattice of rank \(r\).
Define
\begin{equation}
	s_{k}^{\Lambda}(z) = \sum_{\lambda \in \Lambda} \frac{1}{(z-\lambda)^{k}},
\end{equation}
a meromorphic function on \(C_{\infty}\). Then
\begin{equation}
	s_{1}^{\Lambda}(z) = \frac{1}{e^{\Lambda}(z)},
\end{equation}
and the following important property holds. (All of this is shown e.g. in \cite{Gekeler88} Sections 2 and 3.)

\begin{Proposition}\label{Proposition.Characterization-Goss-polynomials}
	There exists a sequence \(G_{k}(X) = G_{k,\Lambda}(X)\) of polynomials over \(C_{\infty}\) (\(k=1,2,3,\dots\)), the \textbf{Goss polynomials of}
	\(\Lambda\), such that:
	\begin{enumerate}[label=\(\mathrm{(\roman*)}\)]
		\item \(s_{k}(z) = G_{k}(s_{1}(z))\);
		\item \(G_{k}\) is monic of degree \(k\);
		\item \(G_{k}(0) = 0\);
		\item \(G_{k}(X) = X^{k}\) if \(k \leq q\);
		\item \(G_{pk} = (G_{k})^{p}\) (\(p = \Characteristic(\mathds{F})\));
		\item \(X^{2}G_{k}'(X) = k G_{k+1}(X)\);
		\item with \(G_{k} = 0\) for \(k \leq 0\), the recursion holds:
		\[
			G_{k}(X) = X(G_{k-1}(X) + \alpha_{1}G_{k-q}(X) + \alpha_{2}G_{k-q^{2}}(X) + \cdots );
		\]
		\item if \(k = q^{j} - 1\) then
		\[
			G_{k}(X) = \sum_{0 \leq i < j} \beta_{i}X^{q^{j}-q^{i}};
		\]
		\item if \(\Lambda\) has finite \(\mathds{F}\)-dimension \(m\) then \(G_{k}(X)\) is divisible by \(X^{n}\) with \(n = [k/q^{m}] + 1\).
	\end{enumerate}	
\end{Proposition}

We will have use for the Goss polynomials when \(\Lambda\) is either the \(n\)-torsion submodule of or the lattice associated with a Drinfeld module.

\subsection{} We write \(\boldsymbol{\omega} \in \Omega\) as \((\omega_{1}, \boldsymbol{\omega}')\) with 
\(\boldsymbol{\omega}' \in \Omega' = \Omega^{r-1}\), so \(\boldsymbol{\omega}' = 1\) if \(r=2\). Accordingly, \(\phi = \phi^{\boldsymbol{\omega}}\)
(resp. \(\phi' = \phi^{\boldsymbol{\omega}'}\)) will be the Drinfeld module of rank \(r\) (resp. \(r-1\)) associated with \(\boldsymbol{\omega}\)
(resp. \(\boldsymbol{\omega}'\)), corresponding to \(\Lambda = \Lambda_{\boldsymbol{\omega}}\) (resp. \(\Lambda' = \Lambda_{\boldsymbol{\omega}'} = \sum_{2 \leq i \leq r} A\omega_{i}\)). So \(\phi\) (resp. \(\phi'\)) is the \enquote{generic} Drinfeld module varying over \(\Omega\)
(resp. \(\Omega'\)).

Fix a \textbf{monic} (i.e., the leading coefficient equals \(1\)) element \(n\) of \(A\) of degree \(d \geq 0\). We define a polynomial \(S_{n}(X)\),
whose coefficients are weak modular forms in \(\boldsymbol{\omega}'\) (if \(r > 2\)), or are constant if \(r = 2\).

\subsection{} Let \(\phi_{n}'(X) = nX + \prescript{}{n}\ell_{1}X^{q} + \cdots + \prescript{}{n}\ell_{(r-1)d}X^{q^{(r-1)d}}\) be the \(n\)-th operator
polynomial of \(\phi'\). As is easily seen, the top coefficient is 
\[
	\Delta_{n}' \defeq \prescript{}{n}\ell_{(r-1)d} = (\Delta')^{(q^{(r-1)d}-1)/(q^{r-1}-1)} \neq 0,
\]
where \(\Delta'\) is the discriminant of \(\phi'\), that is, the leading coefficient of \(\phi_{T}'(X)\). Then
\begin{multline}\label{Eq.Polynomial-expression-for-Sn}
	S_{n}(X) \defeq (\Delta_{n}')^{-1}X^{q^{(r-1)d}} \phi_{n}'(X^{-1}) \\ = 1 + \frac{\prescript{}{n}\ell_{(r-1)d-1}}{\Delta_{n}'} X^{q^{(r-1)d}-q^{(r-1)d-1}} + \cdots + \frac{n}{\Delta_{n}'} X^{q^{(r-1)d}-1}.
\end{multline}
Note that \(S_{n}(0) = 1\), \(\deg S_{n} = q^{(r-1)d}-1\), \(S_{n}\) is sparse (only few non-zero coefficients), and is in fact a polynomial
in \(X^{q-1}\). As announced, the coefficients \(\prescript{}{n}\ell_{i}/\Delta_{n}'\) are weak modular forms in \(\boldsymbol{\omega}'\).

\subsection{}\label{Subsection.n-th-variant-tn-of-uniformizer-of-t}%
Next, we define the \(n\)-th variant \(t_{n}\) of the uniformizer \(t\) of \ref{Subsection.The-uniformizer} as\footnote{This should
not be confused with the quantity \(t_{\mathfrak{n}}\) that appears in \cite{Gekeler25} Section 9. While that \(t_{\mathfrak{n}}\) intuitively means
\enquote{taking the \(n\)-th root of \(t\)}, the present \(t_{n}\) is analogous with \enquote{taking the \(n\)-th power of \(t\)}.}
\begin{align}
	t_{n}(\boldsymbol{\omega})	&\defeq t(n\omega_{1}, \boldsymbol{\omega}').
	\intertext{Then}
	t_{n}(\boldsymbol{\omega})	&= (e^{\Lambda'}(n\omega_{1}))^{-1} \label{Eq.tn-omega-as-power-series-in-t} \\
								&= ( \phi_{n}'(e^{\Lambda'}(\omega_{1})))^{-1} \nonumber \\
								&= (\phi_{n}'(t^{-1}))^{-1}	&&(t = t(\boldsymbol{\omega})) \nonumber \\
								&= (\Delta_{n}')^{-1} t^{q^{(r-1)d}}/S_{n}(t). \nonumber
\end{align}
Regarded as a power series in \(t\), it has order \(q^{(r-1)d}\) and weak modular forms in \(\boldsymbol{\omega}'\) as coefficients.

\subsection{}\label{Subsection.t-expansions-of-Eisenstein-series}%
Now we are able to give the \(t\)-expansion of Eisenstein series. Assume that \(k>0\) is divisible by \(q-1\) (otherwise \(E_{k} \equiv 0\)). Then 
\begin{align*}
	E_{k}(\boldsymbol{\omega})	&= \sum_{a \in A} ~ \sideset{}{^{\prime}}\sum_{\mathbf{b} \in A^{r-1}} \frac{1}{(a\omega_{1} + \mathbf{b}\boldsymbol{\omega}')^{k}} 
	\intertext{( \(\displaystyle\sum \sideset{}{^{\prime}} \sum\) means: if \(a = 0\) then \(\mathbf{b} \neq \mathbf{0}\))} 
								&= E_{k}(\boldsymbol{\omega}') + \sum_{a \neq 0} s_{k}^{\Lambda'}(a\omega_{1})
	\intertext{(by definition of \(s_{k}^{\Lambda'}\) with \(\Lambda' = \Lambda_{\boldsymbol{\omega}'}\))}
								&= E_{k}(\boldsymbol{\omega}') - \sum_{a \in A \text{ monic}} s_{k}^{\Lambda'}(a \omega_{1})		&& (\text{as } \sum_{c \in \mathds{F}^{*}} c^{{-}k} = {-}1) \\
								&= E_{k}(\boldsymbol{\omega}') - \sum_{a \text{ monic}} G_{k, \Lambda'}(s_{1}^{\Lambda'}(a \omega_{1}))	&&\text{by \ref{Proposition.Characterization-Goss-polynomials}(i).}	
\end{align*}
Hence, inserting \(t_{a}(\boldsymbol{\omega})\) for \(s_{1}^{\Lambda'}(a\omega_{1})\) yields
\begin{equation}\label{Eq.Ek-omega-as-Ek-omega-prime-minus-Goss-polynomials}
	E_{k}(\boldsymbol{\omega}) = E_{k}(\boldsymbol{\omega}') - \sum_{a \in A \text{ monic}} G_{k, \Lambda'}(t_{a}(\boldsymbol{\omega})).
\end{equation}
Looking at \eqref{Eq.tn-omega-as-power-series-in-t}, this is a formal expansion in \(t\) as wanted. If is easily seen that its convergence radius
in \(\lvert t \rvert\), depending on \(\boldsymbol{\omega}'\), is always strictly positive; it will be specified and made locally uniform
later (see Section \ref{Section.Growth-coefficients}). The constant part with respect to \(t\) is \(E_{k}(\boldsymbol{\omega}')\), the
\enquote{same} Eisenstein series on \(\Omega' = \Omega^{r-1}\).

For the discriminant we have the following formula, which for \(r=2\) (upon suitable normalization involving the discriminant of \(\phi_{T}'\), 
itself an analogue of \((2\pi \imath)^{2}\)) is analogous with Jacobi's product formula \(q \prod_{n \geq 1} (1-q^{n})^{24}\) for the
elliptic discriminant.

\begin{Theorem}[{\cite{Basson17} Corollary 11, \cite{Gekeler25} Theorem 10.13 and (10.17.3)}]\label{Theorem.Product-expression-discriminant-Delta-omega}
	The discriminant \(\Delta(\boldsymbol{\omega})\) may be expanded as a product
	\begin{equation}\label{Eq.Product-expression-discriminant-Delta-omega}
		\Delta(\boldsymbol{\omega}) = {-}( \Delta'(\boldsymbol{\omega}') )^{q} t(\boldsymbol{\omega})^{q-1} \prod_{a \in A \text{ monic}} S_{a}(t(\boldsymbol{\omega}))^{(q^{r}-1)(q-1)}.
	\end{equation}	
\end{Theorem}

Accordingly, the function \(h\) normalized as in \eqref{Eq.Natural-normalization-of-h} has expansion
\begin{equation}
	h(\boldsymbol{\omega}) = (h'(\boldsymbol{\omega}'))^{q} t(\boldsymbol{\omega}) \prod_{a \text{ monic}} S_{a}(t(\boldsymbol{\omega}))^{q^{r}-1}.
\end{equation}
Here \(\Delta'\) and \(h'\) are the rank-\((r-1)\) versions of \(\Delta\) and \(h\) on \(\Omega' = \Omega^{r-1}\). Again, the convergence radius,
which depends on \(\boldsymbol{\omega}'\), is always positive. This question will be dealt with in Section \ref{Section.Growth-coefficients}.

\section{Hecke operators}\label{Section.Hecke-operators}

\subsection{Hecke correspondences} We let \(\mathcal{L} = \mathcal{L}^{r}\) be the set of \(A\)-lattices of rank \(r\) in \(C_{\infty}\). Fix a
prime \(\mathfrak{p} = (\pi)\) of \(A\) of degree \(d\) with monic generator \(\pi\), and write \(P = q^{d}\). For \(0 \leq i \leq r\) we
consider the correspondence \(T_{\mathfrak{p},i}\) of \(\mathcal{L}\) which associates with each \(\Lambda \in \mathcal{L}\) the finite 
collection
\begin{equation}\label{Eq.Association-lattice-finite-collection-T-p-i}
	T_{\mathfrak{p},i}(\Lambda) = \{ \widetilde{\Lambda} \in \mathcal{L} \mid \Lambda \subset \widetilde{\Lambda} \text{ and } \widetilde{\Lambda}/\Lambda \cong \mathds{F}_{\mathfrak{p}}^{i} \}.
\end{equation}
Trivial properties are
\begin{align*}
	T_{\mathfrak{p},0}(\Lambda)					&= \{ \Lambda \}\\
	T_{\mathfrak{p},r}(\Lambda)					&= \{ \mathfrak{p}^{-1}\Lambda \} \\
	\lvert T_{\mathfrak{p},i}(\Lambda) \rvert 	&= \lvert \Gr_{r,i}(\mathds{F}_{\mathfrak{p}}) \rvert \eqdef c_{r,i}(\mathfrak{p})	&&(0 \leq i \leq r),
\end{align*}
where \(\Gr_{r,i}(\mathds{F}_{\mathfrak{p}})\) is the Graßmannian of \(i\)-subspaces of \(\mathds{F}_{\mathfrak{p}}^{r}\). The size of the latter
is given by \cite{Shimura71} Proposition 3.18 (insert \(P\) for \(p\)):
\begin{equation}
	c_{r,i}(\mathfrak{p}) = \frac{(P^{r}-1)(P^{r}-P) \cdots (P^{r} - P^{i-1})}{(P^{i}-1)(P^{i}-P) \cdots (P^{i} - P^{i-1})}.
\end{equation}
Note that 
\begin{align}
	c_{r,i}(\mathfrak{p})	&\equiv 1 \pmod{p} \quad \text{and} \\
	c_{r,i}(\mathfrak{p})	&= c_{r-1,i}(\mathfrak{p}) + P^{r-i}c_{r-1,i-1}(\mathfrak{p}) \quad \text{for}\quad 1 \leq i < r.\label{Eq.Recursion-c-r-i}	
\end{align}
We regard the \(T_{\mathfrak{p},i}\) as endomorphisms on the free abelian group \(\mathds{Z}[\mathcal{L}]\) of divisors on \(\mathcal{L}\). Then
the \(T_{\mathfrak{p},i}\) (\(\mathfrak{p}\) fixed) commute, \(T_{\mathfrak{p},0} = \id\), and the \(T_{\mathfrak{p},i}\) (\(1 \leq i \leq r\))
are algebraically independent. This follows as in \cite{Shimura71} Theorem 3.20. We let
\[
	\mathbf{H}_{\mathfrak{p}} = \mathds{Z}[T_{\mathfrak{p},i} \mid 1 \leq i \leq r]
\]
be the \textbf{local Hecke algebra at} \(\mathfrak{p}\). By the above, it is a polynomial ring over \(\mathds{Z}\) in the \(T_{\mathfrak{p},i}\).
For different primes \(\mathfrak{p}\), \(\mathfrak{q}\), the \(T_{\mathfrak{p},i}\) and the \(T_{\mathfrak{q},j}\) commute; furthermore, all the
\(T_{\mathfrak{p},i}\) (\(\mathfrak{p}\) a prime, \(1 \leq i \leq r\)) are algebraically independent. Therefore, the global Hecke algebra 
\(\mathbf{H} = \mathds{Z}[T_{\mathfrak{p},i} \mid \mathfrak{p} \text{ prime}, 1 \leq i \leq r]\) is a polynomial ring in infinitely many variables.
It will however play no role in our considerations.

\subsection{Oriented lattices} An \textbf{orientation} on the \(r\)-lattice \(\Lambda\) is the choice of an ordered \(A\)-basis 
\(B = \{\lambda_{1}, \dots, \lambda_{r}\}\) of \(\Lambda\) up to the action of \(\SL(r,A)\). The set \(O(\Lambda)\) of orientations
\([B]\) on \(\Lambda\) is a torsor under \(\GL(r,A)/\SL(r,A) \cong \mathds{F}^{*}\); hence there are precisely \(q-1\) orientations on \(\Lambda\).
We put \(\mathcal{L}^{\pm} = \mathcal{L}^{r,\pm}\) for the set of oriented \(r\)-lattices \((\Lambda, [B])\). The multiplicative group 
\(C_{\infty}^{*}\) acts on \(\mathcal{L}\) and \(\mathcal{L}^{\pm}\), and the diagram
\begin{equation}
	\begin{tikzcd}[row sep=small]
		\SL(r,A) \backslash \Omega \ar[r, "\cong"]					& \mathcal{L}^{\pm}/C_{\infty}^{*} \\
		\text{class of } \boldsymbol{\omega} \ar[r, mapsto]	\ar[d]	& \text{class of \(\Lambda_{\boldsymbol{\omega}}\) with basis \(\{\omega_{1}, \dots, \omega_{r}\}\)} \ar[d] \\
		\Gamma \backslash \Omega \ar[r, "\cong"]						& \mathcal{L}/C_{\infty}^{*}
	\end{tikzcd}
\end{equation}
with natural maps is commutative.

\begin{Lemma}
	Let \(\Lambda, \widetilde{\Lambda} \in \mathcal{L}\) be commensurable (i.e., both finite over \(\Lambda \cap \widetilde{\Lambda}\)). The
	sets \(O(\Lambda)\) and \(O(\widetilde{\Lambda})\) of orientations correspond canonically to each other.	
\end{Lemma}

\begin{proof}
	Assume that \(\Lambda \subset \widetilde{\Lambda}\) is such that the index is prime, i.e., 
	\(\widetilde{\Lambda}/\Lambda \cong \mathds{F}_{\mathfrak{p}}\) for some prime \(\mathfrak{p}\) of \(A\). Choose, by the elementary 
	divisor theorem, an ordered basis \(B = \{\lambda_{1}, \dots, \lambda_{r}\}\) of \(\Lambda\) such that 
	\begin{equation}\label{Eq.Basis-B-tilde-of-lattice-Lambda-tilde}
		\widetilde{B} = \{ \pi^{-1}\lambda_{1}, \lambda_{2}, \dots, \lambda_{r}\} \text{ is a basis of } \widetilde{\Lambda},
	\end{equation}	
	where \(\pi \in \mathfrak{p}\) is the monic generator.
	
	Any other ordered basis \(B'\) of \(\Lambda\) with \([B] = [B']\) and such that \eqref{Eq.Basis-B-tilde-of-lattice-Lambda-tilde} holds with
	\(B'\) and \( (\widetilde{B'}) \) is obtained from \(B\) by base change with some \(\gamma \in \SL(r,A) \cap \Gamma_{0}(\mathfrak{p})\), where
	\(\Gamma_{0}(\mathfrak{p}) \subset \Gamma\) is the subgroup of elements congruent to 
	\begin{center}
		\begin{tikzpicture}[scale=0.75]
			\draw (0,0) rectangle (3,3);
			\draw (1,0) -- (1,3);
			\draw (0,2) -- (3,2);
		
			\node (s1) at (0.5,2.5) {\(*\)};
			\node (s2) at (2,2.5) {\(*\)};
			\node (s3) at (2,1) {\(*\)};
			\node (dots) at (0.5,1) {\(\begin{matrix} 0 \\ \vdots \\ 0 \end{matrix}\)};
			\node (pmod) at (4.5,1.5) {\(\pmod{\mathfrak{p}}\),};
		\end{tikzpicture} 
	\end{center} 
	and then \([\widetilde{B}] = [\widetilde{B'}]\). Hence \([B] \mapsto [\widetilde{B}]\) is a well-defined and bijective map from \(O(\Lambda)\) 
	to \(O(\widetilde{\Lambda})\). Composing, we thus get bijections \(O(\Lambda) \overset{\cong}{\to} O(\widetilde{\Lambda})\) for any pair 
	\(\Lambda \subset \widetilde{\Lambda}\) in \(\mathcal{L}\), which are easily verified to be independent of the chain 
	\(\Lambda \hookrightarrow \Lambda_{1} \hookrightarrow \cdots \hookrightarrow \Lambda_{n} = \widetilde{\Lambda}\) with successively prime 
	steps used for its definition. Thus the assertion follows.
\end{proof}

\subsection{} By the lemma, the correspondence \(T_{\mathfrak{p},i}\) canonically extends to \(\mathcal{L}^{\pm}\); that is, we get a commutative 
diagram
\begin{equation}
	\begin{tikzcd}
		\mathds{Z}[\mathcal{L}^{\pm}] \ar[r, "T_{\mathfrak{p},i}"] \ar[d]	& \mathds{Z}[\mathcal{L}^{\pm}] \ar[d] \\
		\mathds{Z}[\mathcal{L}] \ar[r, "T_{\mathfrak{p},i}"]					& \mathds{Z}[\mathcal{L}].
	\end{tikzcd}
\end{equation}
The upper \(T_{\mathfrak{p},i}\) associates to each \((\Lambda, [B])\in \mathcal{L}^{\pm}\) the collection \((\widetilde{\Lambda}, [\widetilde{B}])\)
with \(\widetilde{\Lambda} \in T_{\mathfrak{p},i}(\Lambda)\), and \([\widetilde{B}]\) is the orientation on \(\widetilde{\Lambda}\) induced by \([B]\).

\subsection{}\label{Subsection.Correspondence-weak-modular-forms-and-functions-on-lattices}%
We will use the following correspondence, well-known in the \enquote{classical} case, between weak modular forms and functions on 
\(\mathcal{L}\) or \(\mathcal{L}^{\pm}\).

Let \(k \in \mathds{N}_{0}\) and \(\ell \in \mathds{Z}/(q-1)\) be such that \(k = r\ell \pmod{q-1}\). There are canonical 1-1-correspondences
between the following sets of \(C_{\infty}\)-valued functions: 
\begin{enumerate}[label=(\alph*)]
	\item functions \(F\) of weight \(k\) and type \(\ell\) on \(\mathcal{L}^{\pm}\); that is, \(F\) satisfies
	\begin{align*}
		F(c\Lambda, [cB])		&= c^{{-}k} F(\Lambda, [B]) \\
		F(\Lambda, [\gamma B])	&= (\det \gamma)^{{-}\ell} F(\Lambda, [B]) \quad\text{for } c \in C_{\infty}^{*}, \gamma \in \Gamma;	
	\end{align*}
	\item functions \(f\) on \(\Omega\) of weight \(k\) and type \(\ell\);
	\item functions \(\widetilde{f}\) on \(\Psi\) with \(\widetilde{f}(c\boldsymbol{\omega}) = c^{{-}k} \widetilde{f}(\boldsymbol{\omega})\) and
	\(\widetilde{f}(\gamma \boldsymbol{\omega}) = (\det \gamma)^{{-}\ell} \widetilde{f}(\boldsymbol{\omega})\).
\end{enumerate}

If \(\ell = 0\) then the function \(F\) in (a) descends to a function on \(\mathcal{L}\) with \(F(c\Lambda) = c^{{-}k}F(\Lambda)\). The translation
is as follows.
\begin{description}
	\item[\((F \leftrightarrow f)\)] \(f(\boldsymbol{\omega}) \defeq F(\Lambda_{\boldsymbol{\omega}}, [\boldsymbol{\omega}])\), where 
	\([\boldsymbol{\omega}]\) is the orientation defined by the basis \(\{\omega_{1}, \dots, \omega_{r}\}\);
	
	\(F(\Lambda, [B]) \defeq \omega_{1}^{{-}k}f(\omega_{1}^{{-}1}\boldsymbol{\omega})\) for \(\Lambda = \sum A\omega_{i}\) with basis
	\(\{\omega_{1}, \dots, \omega_{r}\}\).
	\item[\((f \leftrightarrow \widetilde{f})\)] 	\(\widetilde{f}(\boldsymbol{\omega}) \defeq \omega_{1}^{{-}k}f(\omega_{1}^{{-}1}\boldsymbol{\omega})\)
	
	\(f \defeq\) restriction of \(\widetilde{f}\) to \(\Omega \hookrightarrow \Psi\).
\end{description}

We often do not distinguish these interpretations and write e.g. \enquote{\(\Lambda\)} or \enquote{\((\Lambda, [B])\)} as the argument of a
modular form.

\subsection{} Now we transport the Hecke operators \(T_{\mathfrak{p},i}\) from functions on (oriented) lattices to modular forms by means of
\ref{Subsection.Correspondence-weak-modular-forms-and-functions-on-lattices}. It is obvious that \(T_{\mathfrak{p},i}\)
\begin{itemize}
	\item preserves holomorphy of functions;
	\item maps weak modular forms of type \((k, \ell)\) to such;
	\item preserves the boundary condition \eqref{Eq.Characterization-Drinfeld-modular-form-(ii)};
\end{itemize}
hence it is an operator on the space \(M_{k,\ell}\). Furthermore, the condition that \(f \in M_{k,\ell}\) is a cusp form (i.e., its \(t\)-expansion
is divisible by \(t\)) means that \(f\) vanishes at the boundary of \(\mathcal{M}^{r}(C_{\infty})\) in 
\(\overline{\mathcal{M}}^{r}(C_{\infty}) = \mathds{P}(C_{\infty})\) (see \ref{Subsection.The-moduli-scheme}).

As this condition is also preserved by \(T_{\mathfrak{p},i}\), it maps \(S_{k,\ell}\) to itself.

\begin{Example}
	Suppose that \(r = 2\), and write \(\boldsymbol{\omega} = (\omega, 1)\). For a given \(\mathfrak{p}\) of degree \(d\), the Hecke operators
	are \(T_{\mathfrak{p}} \defeq T_{\mathfrak{p},1}\) and \(T_{\mathfrak{p},2}\). The action of \(T_{\mathfrak{p},2}\) on \(M_{k, \ell}\) is
	\begin{equation}
		T_{\mathfrak{p},2} f(\boldsymbol{\omega}) = f(\pi^{-1}\boldsymbol{\omega}) = \pi^{k} f(\boldsymbol{\omega}),
	\end{equation}	
	and is uninteresting. The action of \(T_{\mathfrak{p}}\) may be described by
	\begin{equation}\label{Eq.Action-of-T-p}
		T_{\mathfrak{p}}f(\omega) = \pi^{k}f(\pi \omega) + \sum_{\substack{b \in A \\ \deg b < d}} f\left( \frac{\omega + b}{\pi} \right),
	\end{equation}
	as the \(q^{d}+1\) sets \(\{ \omega, \pi^{{-}1}\}\), \(\{ \frac{\omega + b}{\pi}, 1\}\) (\(b\) as above) are bases for the \(q^{d} +1\)
	lattices \(\widetilde{\Lambda}\) with \(\widetilde{\Lambda}/\Lambda_{(\omega, 1)} \cong \mathds{F}_{\mathfrak{p}}\).
\end{Example}

\begin{Remarks}
	\begin{enumerate}[wide, label=(\roman*)]
		\item It becomes much more complicated and unpleasant to write (and work with) formulas similar to \eqref{Eq.Action-of-T-p} in the case
		\(r > 2\), since representatives of certain double cosets must be chosen, see \cite{BassonBreuerPink24} Section 12 for more details. Instead we will throughout
		work with functions on lattices, which in our case is simpler and easier to handle.
		\item Why did we define the Hecke correspondence in \eqref{Eq.Association-lattice-finite-collection-T-p-i} via super-lattices 
		\(\widetilde{\Lambda} \supset \Lambda\) and not through sub-lattices \(\Lambda^{\#} \subset \Lambda\)? As long as the base ring \(A\) is
		a principal ideal domain (as in our case), this is merely a question of normalization, since the \(\Lambda^{\#}\) correspond to the
		\(\widetilde{\Lambda}\) through \(\Lambda^{\#} = \pi \widetilde{\Lambda}\). We chose our definition 
		\eqref{Eq.Association-lattice-finite-collection-T-p-i} since it leads to \(A\)-integral formulas. For example, \eqref{Eq.Action-of-T-p} becomes
		\begin{equation}
			T_{\mathfrak{p}}f(\omega) = f(\pi \omega) + \pi^{{-}k} \sum_{b} f\left( \frac{\omega + b}{\pi} \right) \tag{\(3.7.2^{\#}\)}
		\end{equation}
		once the definition of \(T_{\mathfrak{p},i}\) is based on sub-lattices \(\Lambda^{\#}\). In the classical case of elliptic modular forms,
		a formula corresponding to 
		\begin{equation}
			T_{\mathfrak{p}}f(\omega) = \pi^{k-1} f(\pi \omega) + \pi^{{-}1} \sum_{b} f\left( \frac{\omega + b}{\pi} \right) \tag{3.7.2'}
		\end{equation}
		is used, as then Hecke eigenvalues and certain Fourier coefficients become equal for normalized newforms \(f\). This is however meaningless
		in our framework.
	\end{enumerate}	
\end{Remarks}

\begin{Example}\label{Example.Action-of-T-p-i-on-Eisenstein-series-Ek}%
	We calculate the action of \(T_{\mathfrak{p},i}\) on the Eisenstein series \(E_{k}\). Now 
	\begin{align}
		(T_{\mathfrak{p},i} E_{k})(\boldsymbol{\omega})	&= (T_{\mathfrak{p},i}E_{k})(\Lambda_{\boldsymbol{\omega}}) \nonumber \\
														&= \sum_{\widetilde{\Lambda} \supset \Lambda, \widetilde{\Lambda}/\Lambda \cong \mathds{F}_{\mathfrak{p}}^{i}} E_{k}(\widetilde{\Lambda}) \nonumber \\
														&= \sideset{}{^{\prime}} \sum_{\lambda \in \mathfrak{p}^{{-}1}\Lambda_{\boldsymbol{\omega}}} \nu(\lambda) \lambda^{{-}k}	\nonumber 
		\intertext{with \(\nu(\lambda) = \lvert \{ \widetilde{\Lambda} \mid \lambda \in \widetilde{\Lambda} \} \rvert \)}
														&= c_{r,i}(\mathfrak{p}), \text{ if } \lambda \in \Lambda_{\boldsymbol{\omega}} \\
														&= c_{r-1,i-1}(\mathfrak{p}), \text{ if } \lambda \notin \Lambda_{\boldsymbol{\omega}}, \nonumber
	\end{align}
	as in the latter case, the class of \(\lambda\) in \(\mathfrak{p}^{{-}1}\Lambda_{\boldsymbol{\omega}}/\Lambda_{\boldsymbol{\omega}}\) is
	contained in precisely \(c_{r-1,i-1}(\mathfrak{p})\) hyperplanes of dimension \(i\) of the \(\mathds{F}_{\mathfrak{p}}\)-space
	\(\mathfrak{p}^{{-}1}\Lambda_{\boldsymbol{\omega}}/\Lambda_{\boldsymbol{\omega}}\). In either case, \(\nu(\lambda) \equiv 1 \pmod{p}\), and
	thus
	\begin{equation}
		(T_{\mathfrak{p},i} E_{k})(\boldsymbol{\omega}) = E_{k}(\pi^{{-}1} \Lambda_{\boldsymbol{\omega}}) = \pi^{k} E_{k}(\boldsymbol{\omega}).
	\end{equation}
	Hence \(E_{k}\) is an eigenform with eigenvalue \(\pi^{k}\) of \(T_{\mathfrak{p},i}\), regardless of \(i\).
\end{Example}

Hecke operators behave simply under \(p\)-th powers, \(p = \Characteristic(\mathds{F})\).

\begin{Proposition}\label{Proposition.Identity-for-C-infty-valued-functions-on-L-or-Lpm}%
	For any \(C_{\infty}\)-valued function \(f\) on \(\mathcal{L}\) or \(\mathcal{L}^{\pm}\),
	\[
		T_{\mathfrak{p},i}(f^{p}) = (T_{\mathfrak{p},i} f)^{p}
	\]
	holds.
\end{Proposition}

\begin{proof}
	\[
		(T_{\mathfrak{p},i}f^{p})(\Lambda) = \sum_{\widetilde{\Lambda} \supset \Lambda, \widetilde{\Lambda}/\Lambda \cong \mathds{F}_{\mathfrak{p}}^{i}} f^{p}(\widetilde{\Lambda}) = \Big( \sum_{\widetilde{\Lambda}} f(\widetilde{\Lambda}) \Big)^{p} = (T_{\mathfrak{p},i} f(\Lambda))^{p}.
	\]	
\end{proof}

\section{The effect of Hecke operators on \(t\)-expansions}\label{Section.The-effect-of-Hecke-operators-on-t-expansions}

We keep the notations and assumptions of the last section.

\subsection{} First, we compare the Hecke correspondence \(T_{\mathfrak{p},i} = T_{\mathfrak{p},i}^{(r)}\) defined on \(\mathcal{L} = \mathcal{L}^{r}\)
as in \eqref{Eq.Association-lattice-finite-collection-T-p-i} with the correspondences \(T_{\mathfrak{p},j}^{(r-1)}\) defined on the boundary,
that is, on \(\mathcal{L}^{r-1}\). We assume \(r \geq 3\), as the case \(r=2\) has been settled in \cite{Gekeler88} 7.3.

A splitting of the \(r\)-lattice \(\Lambda\) is the choice of a basis vector \(\lambda_{1}\) and a direct complement:
\begin{equation}
	\Lambda = A\lambda_{1} \oplus \Lambda'
\end{equation}
with \(\Lambda' \in \mathcal{L}^{r-1}\). We write \(\mathcal{L}^{\perp} = \mathcal{L}^{r,\perp}\) for the set of \(r\)-lattices provided with a
splitting \((\lambda_{1}, \Lambda')\). If \(\Lambda\) is presented as \(\Lambda_{\boldsymbol{\omega}}\), its \textbf{induced splitting} is
given by \(\lambda_{1} = \omega_{1}\) and \(\Lambda' = \Lambda_{\boldsymbol{\omega}'} = \sum_{2 \leq i \leq r} A\omega_{i}\).

\subsection{}\label{Subsection.Lambda-tilde-super-lattice}%
Let \(\widetilde{\Lambda}\) be a super-lattice of \(\Lambda = \Lambda_{\boldsymbol{\omega}}\) such that 
\(\widetilde{\Lambda}/\Lambda \cong \mathds{F}_{\mathfrak{p}}^{i}\) with \(1 \leq i < r\). Then 
\(\widetilde{\Lambda}' = (\widetilde{\Lambda})' \defeq \widetilde{\Lambda} \cap \sum_{2 \leq i \leq r} K\omega_{i}\) is a super-lattice 
of \(\Lambda' = \Lambda_{\boldsymbol{\omega}}\) and 
\begin{align}
	\text{either } \widetilde{\Lambda}'/\Lambda'		&\cong \mathds{F}_{\mathfrak{p}}^{i}	 	&&(\text{type 1}) \\
	\text{or } \widetilde{\Lambda}'/\Lambda'			&\cong \mathds{F}_{\mathfrak{p}}^{i-1}	&&(\text{type 2}) \nonumber	
\end{align}
holds. When \(\widetilde{\Lambda}\) varies over \(T_{\mathfrak{p},i}(\Lambda)\), the former occurs \(c_{r-1,i}(\mathfrak{p})\) times, while the
latter appears \(P^{r-i}c_{r-1,i-1}(\mathfrak{p})\) often, since each such \(\widetilde{\Lambda}'\) has precisely \(P^{r-1-(i-1)} = P^{r-i}\) 
extensions \(\widetilde{\Lambda}\) with \(\widetilde{\Lambda}/\Lambda \cong \mathds{F}_{\mathfrak{p}}^{i}\) (a precise description is given in
\ref{Subsection.Evaluation-of-4-4-3-terms-of-type-2}). Note that this provides a combinatorial interpretation of the rule \eqref{Eq.Recursion-c-r-i}.
There results the important identity of divisors on \(\mathcal{L}^{r-1}\):

\begin{Proposition}\label{Proposition.T-p-i-on-lattice}
	\[
		\big(T_{\mathfrak{p},i}^{(r)}(\Lambda_{\boldsymbol{\omega}}) \big)' = T_{\mathfrak{p},i}^{(r-1)}(\Lambda_{\boldsymbol{\omega}'}) + P^{r-i} T_{\mathfrak{p},i-1}^{(r-1)}(\Lambda_{\boldsymbol{\omega}'})
	\]	
\end{Proposition}

Here the left hand side is the divisor of \((~)'\)-parts of the various \(\widetilde{\Lambda}_{\boldsymbol{\omega}}\) that appear in
\(T_{\mathfrak{p},i}(\Lambda_{\boldsymbol{\omega}})\).

As \(P=0\) in \(C_{\infty}\), the second term of the right hand side vanishes whenever we evaluate \ref{Proposition.T-p-i-on-lattice} on
\(C_{\infty}\)-valued functions.

\subsection{}\label{Subsection.Relation-t-expansion-of-f-to-expansion-of-T-p-i-f}%
Let now \(f \in M_{k,\ell}\) have \(t\)-expansion
\begin{equation}\stepcounter{subsubsection}\label{Eq.t-expansion-of-f}%
	f(\boldsymbol{\omega}) = \sum_{n \geq 0} a_{n}(\boldsymbol{\omega}') t^{n}(\boldsymbol{\omega}).
\end{equation}
What can we say about the expansion of \(T_{\mathfrak{p},i}f\)? First note that the function \(t\) is not defined on the set \(\mathcal{L}\)
of lattices, but on the set \(\mathcal{L}^{\perp}\) of lattices provided with a splitting. If 
\(\widetilde{\Lambda}\in T_{\mathfrak{p},i}(\Lambda_{\boldsymbol{\omega}})\) then its induced splitting is as in 
\ref{Subsection.Lambda-tilde-super-lattice}, viz:
\subsubsection{}\label{Subsubsection.Induced-splitting}\stepcounter{equation}%
\((\widetilde{\omega}_{1}, \widetilde{\Lambda}')\) with \(\widetilde{\Lambda}' = \widetilde{\Lambda} \cap \sum_{2 \leq i \leq r} K\omega_{i}\)
and 
\begin{itemize} 
	\item \(\widetilde{\omega}_{1} = \omega_{1}\) if \(\widetilde{\Lambda}\) is of type 1.	
	\item \(\widetilde{\omega}_{1}\) to be described in \ref{Subsection.Evaluation-of-4-4-3-terms-of-type-2} if \(\widetilde{\Lambda}\) is of type 2.
\end{itemize}

Then \((T_{\mathfrak{p},i}f)(\boldsymbol{\omega})\) may be written as
\begin{align}
	(T_{\mathfrak{p},i}f)(\boldsymbol{\omega}) 	&= \sum_{\widetilde{\Lambda} \in T_{\mathfrak{p},i}(\Lambda_{\boldsymbol{\omega}})} \sum_{n \geq 0} a_{n}(\widetilde{\Lambda}')t^{n}(\widetilde{\Lambda}) \label{Eq.Expansion-of-T-p-i-f} \\
												&= \sum_{n \geq 0} \sum_{\widetilde{\Lambda}} a_{n}(\widetilde{\Lambda}')t^{n}(\widetilde{\Lambda}), \nonumber 
\end{align}
where \(\widetilde{\Lambda}\) carries its induced splitting and \(a_{n}(\widetilde{\Lambda}')\) is the weak modular form in \(\boldsymbol{\omega}'\)
that occurs in \eqref{Eq.t-expansion-of-f}, evaluated on the \((r-1)\)-lattice \(\widetilde{\Lambda}'\). Hence we are reduced to investigating
the inner sums, where we distinguish between terms for \(\widetilde{\Lambda}\) of type 1 or 2.

\subsection{} For \(\widetilde{\Lambda}\) as above, consider the inclusions 
\begin{equation}\stepcounter{subsubsection}\label{Eq.Inclusions-based-on-superlattice}%
	\Lambda' = \Lambda_{\boldsymbol{\omega}'} \subset \widetilde{\Lambda}' \subset \pi^{{-}1}\Lambda_{\boldsymbol{\omega}'} \subset \pi^{{-}1} \widetilde{\Lambda}'.
\end{equation}
As \(\mathds{F}_{\mathfrak{p}}\)-spaces, the dimensions of \(\widetilde{\Lambda}'/\Lambda_{\boldsymbol{\omega}'}\) and 
\(\pi^{{-}1}\Lambda_{\boldsymbol{\omega}'}/\widetilde{\Lambda}'\) are \((i,r-1-i)\) for type 1 and \((i-1,r-i)\) for type 2.

Let \(\phi^{\Lambda_{\boldsymbol{\omega}'}}\) and \(\phi^{\widetilde{\Lambda}'}\) be the respective Drinfeld modules of rank \(r-1\), associated
with these lattices and 

\subsubsection{}\stepcounter{equation}%
\(\varphi = \varphi^{\widetilde{\Lambda}'|\Lambda_{\boldsymbol{\omega}'}}\) the isogeny from \(\phi^{\widetilde{\Lambda}'}\) to 
\(\phi^{\Lambda_{\boldsymbol{\omega}'}}\) induced from \(\Lambda_{\boldsymbol{\omega}'} \hookrightarrow \widetilde{\Lambda}'\), normalized
with derivative 1.

Then for the corresponding exponential functions,
\begin{equation}\label{Eq.Identity-for-corresponding-exponential-functions}
	e^{\widetilde{\Lambda}'} = \varphi(e^{\Lambda_{\boldsymbol{\omega}'}})
\end{equation}
holds. Specifying \(\widetilde{\Lambda}'\) is the same as specifying an \(i\)-dimensional (type 1) or an \((i-1)\)-dimensional (type 2)
\(\mathds{F}_{\mathfrak{p}}\)-subspace \(H = H(\widetilde{\Lambda}')\) of the \((r-1)\)-dimensional \(\mathds{F}_{\mathfrak{p}}\)-space
\(\prescript{}{\mathfrak{p}} \phi^{\Lambda_{\boldsymbol{\omega}'}}\) of \(\mathfrak{p}\)-division points of 
\(\phi^{\Lambda_{\boldsymbol{\omega}'}}\). In fact, regarded as an additive polynomial, \(\varphi\) is the exponential function of \(H\):
\begin{equation}\label{Eq.Identity-for-phi-exponential-function-on-H}
	\varphi(X) = X \sideset{}{^{\prime}} \prod_{h \in H} (1 - X/h).
\end{equation}
It has shape
\[
	\varphi(X) = X + c_{1}X^{q} + \cdots + c_{d \cdot \dim H} X^{P^{\dim H}}
\]
(\(P = q^{d} = q^{\deg \mathfrak{p}}\), and \(\dim H = i\) or \(i-1\)).

The leading coefficient \(\Delta_{\varphi} \defeq c_{d \cdot \dim H}\) is non-vanishing and equals the product 
\((\sideset{}{^{\prime}}\prod_{h \in H} h)^{-1}\). Let 
\begin{equation}
	S_{H}(X) \defeq \Delta_{\varphi}^{{-}1} X^{q^{d \cdot \dim H}} \varphi(X^{{-}1}) = \sideset{}{^{\prime}} \prod_{h \in H} (1 - hX) = 1 + o(X^{q^{d \cdot \dim H-1}(q-1)})
\end{equation}
be the reciprocal polynomial of \(\varphi\). Like \(\varphi\) and \(H\), it depends on \(\widetilde{\Lambda}'\) and will be labelled 
\(S_{H} = S_{\widetilde{\Lambda}'}\) if the need arises. (The reader will notice the similarity of construction and usage of \(S_{H}\) with
that of the polynomials \(S_{n}\) of \eqref{Eq.Polynomial-expression-for-Sn}.)

\subsection{Evaluation of \eqref{Eq.Expansion-of-T-p-i-f}: terms of type 1}\label{Subsection.Evaluation-of-4-4-3-terms-of-type-1}%
Essentially we must express \(t(\widetilde{\Lambda})\) through \(t(\boldsymbol{\omega}) = (e^{\Lambda_{\boldsymbol{\omega}'}}(\omega_{1}))^{{-}1}\).
Suppose that \(\widetilde{\Lambda}\) is of type 1, so its splitting is \((\omega_{1}, \widetilde{\Lambda}')\). Therefore,  
\begin{align*}
	t(\widetilde{\Lambda}) 	&= (e^{\widetilde{\Lambda}'}(\omega_{1}))^{{-}1} \\
							&= \frac{1}{\varphi(e^{\Lambda_{\boldsymbol{\omega}'}}(\omega_{1}))} = \frac{1}{\varphi(t^{{-}1})} = \frac{t^{q^{di}}}{\Delta_{\varphi} \cdot S_{H}(t)} = \Delta_{\varphi}^{{-}1}\big(t^{q^{di}} + o(t^{q^{di} + q^{di-1}(q-1)})\big).
\end{align*}
We now label \(\Delta_{\varphi}\) as \(\Delta_{\widetilde{\Lambda}'}\) and \(S_{H}\) as \(S_{\widetilde{\Lambda}'}\). Then the term 
\(a_{n}(\widetilde{\Lambda}')t^{n}(\widetilde{\Lambda})\) of \eqref{Eq.Expansion-of-T-p-i-f} equals
\begin{equation}
	a_{n}(\widetilde{\Lambda}')t^{n}(\widetilde{\Lambda}) = \frac{a_{n}(\widetilde{\Lambda}')}{\Delta_{\widetilde{\Lambda}'}} \frac{t^{nq^{di}}}{S_{\widetilde{\Lambda}'}(t)}.
\end{equation}
Although we are unable to simplify it further, we can state that as a power series in \(t\), it has shape
\begin{equation}
	a_{n}(\widetilde{\Lambda}')t^{n}(\widetilde{\Lambda}) = C(\widetilde{\Lambda}')t^{nq^{di}} + o( t^{nq^{d i} + q^{di-1}(q-1)}).
\end{equation}
Here \(C(\widetilde{\Lambda}') \neq 0\) is a constant that depends only on \(\widetilde{\Lambda}'\). 

\subsection{Terms of type 2}\label{Subsection.Evaluation-of-4-4-3-terms-of-type-2}%
Fix some \(\widetilde{\Lambda}' \in T_{\mathfrak{p},i-1}(\Lambda')\), \(\Lambda' = \Lambda_{\boldsymbol{\omega}'}\). Let \(\overline{V}\) be 
an \(\mathds{F}_{\mathfrak{p}}\)-vector space complement of \(\widetilde{\Lambda}'/\Lambda'\) in \(\pi^{{-}1}\Lambda'/\Lambda'\), and lift it to 
an \(\mathds{F}\)-subspace \(V\) of \(\pi^{{-}1}\Lambda'\). Then \(\lvert V \rvert = P^{r-i} = q^{d(r-i)}\). Each \(\widetilde{\Lambda}\) 
with \(\widetilde{\Lambda}/\Lambda \cong \mathds{F}_{\mathfrak{p}}^{i}\) and 
\(\widetilde{\Lambda} \cap \sum_{2 \leq i \leq r} K\omega_{i} = \widetilde{\Lambda}'\) is of the shape
\begin{equation}\stepcounter{subsubsection}\label{Eq.Direct-sum-decomposition-of-super-lattice}%
	\widetilde{\Lambda} = A \widetilde{\omega}_{1} \oplus \widetilde{\Lambda}', \text{ where } \widetilde{\omega}_{1} = \pi^{{-}1}\omega_{1} + v \text{ with a well-defined } v \in V.
\end{equation}
Conversely, each such \(\widetilde{\Lambda}\) belongs to \(T_{\mathfrak{p},i}(\Lambda)\) and satisfies 
\(\widetilde{\Lambda} \cap \sum_{2 \leq i \leq r} K\omega_{i} =\) the given \(\widetilde{\Lambda}'\). The induced splitting on 
such a \(\widetilde{\Lambda}\) referred to in \eqref{Subsubsection.Induced-splitting} is 
\subsubsection{}\stepcounter{equation}%
\((\widetilde{\omega}_{1}, \widetilde{\Lambda}')\), and the value of \(t(\widetilde{\Lambda}, \widetilde{\omega}_{1}, \widetilde{\Lambda}')\) 
doesn't depend on the choices of \(\overline{V}\) and \(V\) made.

We further observe:
\begin{equation}\label{Eq.p-division-points-of-phi-widetilde-Lambda-prime}
	W = W(\widetilde{\Lambda}') \defeq e^{\widetilde{\Lambda}'}(V) \subset \prescript{}{\mathfrak{p}} \phi^{\widetilde{\Lambda}'}
\end{equation}
is the submodule of those \(\mathfrak{p}\)-division points of \(\phi^{\widetilde{\Lambda}'}\) which vanish under the isogeny \(\psi\) dual to
\(\varphi\).\footnote{Caution: There is no functorial duality of isogenies of Drinfeld modules; this requires enlarging the theory to Anderson modules \cite{Anderson86}. The present is merely an ad hoc construction.} Here \(\psi\) is such that 
\begin{align}
	\varphi \circ \psi = \pi^{{-}1} \phi_{\pi}^{\Lambda'}	&&(\text{see \eqref{Eq.Inclusions-based-on-superlattice}}).	
\end{align}
In analogy with \eqref{Eq.Identity-for-corresponding-exponential-functions} and \eqref{Eq.Identity-for-phi-exponential-function-on-H} we have
\begin{align}
	e^{\pi^{{-}1}\Lambda'} 	&= \psi(e^{\widetilde{\Lambda}'})\label{Eq.Identity-exponential-function-and-psi}
	\intertext{and}
	\psi(X)					&= X \sideset{}{^{\prime}} \prod_{w \in W} (1-X/w).
\end{align}

\subsection{}\label{Subsection.n-thGoss-polynomials-of-W}%
Let \(G_{n}(X) = G_{n,W}(X)\) be the \(n\)-th Goss polynomial of \(W\). (Recall that it depends on the choice of \(\widetilde{\Lambda}'\).) 
Now we are ready to evaluate \(\sum_{\widetilde{\Lambda} \mid \widetilde{\Lambda}'} t^{n}(\widetilde{\Lambda})\), where the sum is over 
the \(\widetilde{\Lambda}\) as in \eqref{Eq.Direct-sum-decomposition-of-super-lattice}. Namely,
\begin{align}
	\sum_{\widetilde{\Lambda} \mid \widetilde{\Lambda}'} t^{n}(\widetilde{\Lambda})	&= \sum_{v \in V} \frac{1}{e^{\widetilde{\Lambda}'}(\pi^{{-}1}\omega_{1} + v)^{n}} \\
			&= \sum_{w \in W} \frac{1}{(e^{\widetilde{\Lambda}'}(\omega_{1}/\pi) + w)^{n}}	&&\text{by \eqref{Eq.p-division-points-of-phi-widetilde-Lambda-prime}}. \nonumber 
	\intertext{The formal identity}
	\sum_{w \in W} \frac{1}{X-w}	&= \frac{1}{\psi(X)}	
	\intertext{gives}
	\sum_{w \in W} \frac{1}{e^{\widetilde{\Lambda}'}(\omega_{1}/\pi) + w}	&= \frac{1}{\psi(e^{\widetilde{\Lambda}'}(\omega_{1}/\pi))} \\
											&= \frac{1}{e^{\pi^{{-}1}\Lambda'}(\omega_{1}/\pi)}	&&(\text{by \eqref{Eq.Identity-exponential-function-and-psi}}) \nonumber \\
											&= \frac{1}{\pi^{{-}1} e^{\Lambda'}(\omega_{1})} = \pi t(\boldsymbol{\omega}). \nonumber 
\end{align}
We conclude with the defining property of Goss polynomials to find that
\begin{equation}
	\sum_{\widetilde{\Lambda} \mid \widetilde{\Lambda}'} t^{n}(\widetilde{\Lambda}) = G_{n,W}(\pi t(\boldsymbol{\omega})).
\end{equation}
We collect the results so far. 

\begin{Proposition}\label{Proposition.t-expansion-of-f-in-M-k-l-to-power-series-expansion-of-T-p-i-f}
	Let \(f \in M_{k,\ell}\) have \(t\)-expansion
	\begin{align*}
		f(\boldsymbol{\omega}) = \sum_{n \geq 0} a_{n}(\boldsymbol{\omega}') t^{n}(\boldsymbol{\omega}) &&\text{as in \eqref{Eq.t-expansion-of-f}}.
	\end{align*}
	Then as a power series in \(t\), the form \(T_{\mathfrak{p},i}f\) is given by 
	\begin{equation} \label{Eq.t-expansion-of-f-in-M-k-l-to-power-series-expansion-of-T-p-i-f}
		T_{\mathfrak{p},i}f(\boldsymbol{\omega}) = \sum_{n \geq 0} \Big( \sum_{\widetilde{\Lambda}' \text{ of type 1}} \frac{a_{n}(\widetilde{\Lambda}')}{\Delta_{\widetilde{\Lambda}'}} \frac{t^{nq^{di}}}{S_{\widetilde{\Lambda}'}(t)} + \sum_{\widetilde{\Lambda}' \text{ of type 2}} a_{n}(\widetilde{\Lambda}') G_{n}^{\widetilde{\Lambda}'}(\pi t(\boldsymbol{\omega})) \Big).
	\end{equation}
	Here the \(\widetilde{\Lambda}'\) of type 1 (resp. type 2) run through \(T_{\mathfrak{p},i}^{(r-1)}(\Lambda')\) (resp. through
	\(T_{\mathfrak{p},i-1}^{(r-1)}(\Lambda')\)) and \(\Delta_{\widetilde{\Lambda}'}\), \(S_{\widetilde{\Lambda}'}\), and \(G_{n}^{\widetilde{\Lambda}'} = G_{n,H}\) are the quantities determined by \(\widetilde{\Lambda}'\) and described in 
	\ref{Subsection.Evaluation-of-4-4-3-terms-of-type-1} and \ref{Subsection.n-thGoss-polynomials-of-W}.
\end{Proposition}

Some explanation is in order.

\begin{Remarks}
	\begin{enumerate}[wide, label=(\roman*)]
		\item The terms corresponding to \(\widetilde{\Lambda}'\) of type 1 are power series of order \(nq^{di} \gg n\) in 
		\(t(\boldsymbol{\omega})\), while terms corresponding to \(\widetilde{\Lambda}'\) of type 2 are polynomials of degree \(n\).
		The vanishing order of a Goss polynomial \(G_{m}^{\widetilde{\Lambda}'}\) in \(X = 0\) is larger or equal to \([m/q^{d(r-i)}]+1\) by 
		\ref{Proposition.Characterization-Goss-polynomials}(ix). Hence, for a given term \(a_{n} = a_{n}(T_{\mathfrak{p},i}f)\) of the
		expansion of \(T_{\mathfrak{p},i}f\) with \(n > 0\) only finitely many of the 
		\(a_{m}(\widetilde{\Lambda}')G_{m}^{\widetilde{\Lambda}'}(\pi t)\) may contribute. In other words, 
		\eqref{Eq.t-expansion-of-f-in-M-k-l-to-power-series-expansion-of-T-p-i-f} is an identity of formal power series, whose evaluation
		requires no analysis.
		\item For \(f \in M_{k,\ell}\) also \(T_{\mathfrak{p},i}f \in M_{k,\ell}\), and thus the coefficients \(a_{n}(T_{\mathfrak{p},i}f)\)
		as functions in \(\boldsymbol{\omega}'\) are weak modular forms for \(\Gamma' = \GL(r-1,A)\) of weight \(k-n\) and type \(\ell\)
		(\cite{Gekeler25} 7.14, \cite{BassonBreuerPink24} Theorem 5.9). On the other hand, the \(a_{m}(\widetilde{\Lambda}')\) that appear in 
		\eqref{Eq.t-expansion-of-f-in-M-k-l-to-power-series-expansion-of-T-p-i-f} (like the \(\Delta_{\widetilde{\Lambda}'}\) and the
		coefficients of \(S_{\widetilde{\Lambda}'}(X)\) and \(G_{m}^{\widetilde{\Lambda}'}(X)\)) as functions in \(\boldsymbol{\omega}'\)
		are weakly modular only for the congruence subgroup \(\Gamma_{\widetilde{\Lambda}'}'\) of \(\Gamma'\) that preserves
		\(\widetilde{\Lambda}'\). Hence the summation over the \(\widetilde{\Lambda}'\) in 
		\eqref{Eq.t-expansion-of-f-in-M-k-l-to-power-series-expansion-of-T-p-i-f} symmetrizes the coefficients. To simplify the sum, we had
		(at least) to solve the basic problem circumscribed in Problem \ref{Problem.Evaluation-of-Goss-polynomials-in-terms-of-Phi}.
		\item The formula collapses for \(n = 0\) to 
		\begin{equation} 
			a_{0}(T_{\mathfrak{p},i}^{(r)}f)(\boldsymbol{\omega}') = \sum_{\widetilde{\Lambda}' \text{ of type 1}} a_{0}(\widetilde{\Lambda}') = (T_{\mathfrak{p},i}^{(r-1)}a_{0})(\boldsymbol{\omega}'),	
		\end{equation}
		in keeping with Proposition \ref{Proposition.T-p-i-on-lattice} and the result \ref{Example.Action-of-T-p-i-on-Eisenstein-series-Ek}
		about Eisenstein series.
	\end{enumerate}	
\end{Remarks}

\begin{Problem}\label{Problem.Evaluation-of-Goss-polynomials-in-terms-of-Phi}
	Let a Drinfeld module \(\phi\) of rank \(r\) over \(C_{\infty}\) be given, with module \(\prescript{}{\mathfrak{p}}\phi\) of \(\mathfrak{p}\)-division
	points, where \(\mathfrak{p}\) is a prime of \(A\). Evaluate in terms of \(\phi\) the sum of Goss polynomials
	\[
		\sum G_{n,W}(X),
	\]	
	where \(W\) runs through the sub-\(\mathds{F}_{\mathfrak{p}}\)-modules of \(\prescript{}{\mathfrak{p}}\phi\) of 
	\(\mathds{F}_{\mathfrak{p}}\)-dimension \(i\) (\(1 \leq i < r\))! 
	
	As the complexity of \eqref{Eq.t-expansion-of-f-in-M-k-l-to-power-series-expansion-of-T-p-i-f} shows, \(t\)-expansions are not very well 
	adapted to Hecke operators. \(A\)-expansions like \eqref{Eq.Ek-omega-as-Ek-omega-prime-minus-Goss-polynomials} for Eisenstein series do 
	better in this respect, as the next result shows.
\end{Problem}

\begin{Proposition}\label{Proposition.Action-of-T-p-i-on-Goss-polynomials}
	Let \((G_{n, \Lambda'})_{n \in \mathds{N}_{0}}\) be the sequence of Goss polynomials for \(\Lambda' = \Lambda_{\boldsymbol{\omega}}\) and
	\(t_{a}\) (\(a \in A\) monic) the \(a\)-variant of \(t\) as in \ref{Subsection.n-th-variant-tn-of-uniformizer-of-t}. Then \(T_{\mathfrak{p},i}\)
	acts on \(G_{n}(t_{a}) = G_{n, \Lambda'}(t_{a})\) through 
	\begin{align}
		T_{\mathfrak{p},i}(G_{n}(t_{a})) 	&= \pi^{n}G_{n}(t_{a\pi}),	&&a \in \mathfrak{p}, \\
											&= \pi^{n}G_{n}(t_{a\pi}) + \pi^{n}G_{n}(t_{a}),	&&a \notin \mathfrak{p}. \nonumber
	\end{align}	
\end{Proposition}

\begin{proof}
	\begin{enumerate}[wide, label=(\roman*)]
		\item We must evaluate the left hand side on \(\Lambda = \Lambda_{\boldsymbol{\omega}}\), that is, replace \(\Lambda\) by 
		\(T_{\mathfrak{p},i}(\Lambda) = \{ \widetilde{\Lambda} \text{ as in \ref{Subsection.Relation-t-expansion-of-f-to-expansion-of-T-p-i-f}} \}\), where
		each \(\widetilde{\Lambda}\), depending on its type 1 or 2, carries its induced splitting.
		\item By the calculation in \ref{Subsection.t-expansions-of-Eisenstein-series},
		\begin{equation}
			G_{n}(t_{a}(\boldsymbol{\omega})) = \sum_{\mathbf{b} \in A^{r-1}} \frac{1}{(a\omega_{1} + \mathbf{b}\boldsymbol{\omega}')^{n}}.
		\end{equation}
		Therefore,
		\[
			\sum_{\substack{\widetilde{\Lambda} \in T_{\mathfrak{p},i}(\Lambda) \\ \text{of type 1}}} G_{n}(t_{a}(\widetilde{\Lambda})) = \sum_{\mathbf{b} \in \mathfrak{p}^{-1}A^{r-1}} \frac{\nu(\mathbf{b})}{(a \omega_{1} + \mathbf{b}\boldsymbol{\omega}')^{n}},
		\]
		where \(\nu(\mathbf{b})\) is the number of \(A\)-modules \(L \subset \mathfrak{p}^{-1}A^{r-1}\) above \(A^{r-1}\) such that \(L/A^{r-1}\) is
		isomorphic with \(\mathds{F}_{\mathfrak{p}}^{i}\) and \(\mathbf{b} \in L\). As in Example \ref{Example.Action-of-T-p-i-on-Eisenstein-series-Ek},
		\begin{align*}
			\nu(\mathbf{b})	&= c_{r-1,i}(\mathfrak{p}),		&&\text{if } \mathbf{b} \in A^{r-1} \\
							&= c_{r-2,i-1}(\mathfrak{p}),	&&\text{if } \mathbf{b} \notin A^{r-1},
		\end{align*}
		and is \(\equiv 1 \pmod{p}\) anyway. Hence
		\begin{align*}
			\sum_{\widetilde{\Lambda} \text{ of type 1}} G_{n}(t_{a}(\widetilde{\Lambda})) 	&= \sum_{\mathbf{b} \in \mathfrak{p}^{-1}A^{r-1}} \frac{1}{(a\omega_{1} + \mathbf{b}\boldsymbol{\omega}')^{n}} \\
					&= \sum_{\mathbf{b} \in A^{r-1}} \frac{1}{\pi^{n}(a\pi \omega_{1} + \mathbf{b}\boldsymbol{\omega}')^{n}} = \pi^{n} G_{n}(t_{a\pi}(\boldsymbol{\omega})).
		\end{align*}
		\item We evaluate the sum over the \(\widetilde{\Lambda}\) of type 2 in the same manner. These are obtained from \(A\)-modules \(L\) with
		\begin{equation}\label{Eq.Inclusion-of-A-modules}
			A^{r-1} \subset L \subset \mathfrak{p}^{-1}A^{r-1} \quad \text{and} \quad L/A^{r-1} \cong \mathds{F}_{\mathfrak{p}}^{i-1}
		\end{equation}
		as follows. Let
		\[
			L_{\boldsymbol{\omega}'} \defeq \{ \mathbf{b} \boldsymbol{\omega}' \mid \mathbf{b} \in L \}.
		\]
		Fix such an \(L\) and choose an \(\mathds{F}\)-complement \(V_{L}\) of \(L\) in \(\mathfrak{p}^{-1}A^{r-1}\) as in 
		\ref{Subsection.Evaluation-of-4-4-3-terms-of-type-2}. Then the \(\widetilde{\Lambda}\) of type 2 with \((\widetilde{\Lambda}') = L_{\boldsymbol{\omega}'}\) are the
		\begin{equation}
			\widetilde{\Lambda} = A\widetilde{\omega}_{1} \oplus L_{\boldsymbol{\omega}'} \quad \text{with} \quad \widetilde{\omega}_{1} = \pi^{-1}\omega_{1} + \mathbf{v} \boldsymbol{\omega}',
		\end{equation}
		where \(\mathbf{v}\) runs through \(V_{L}\). (This is just another way to state \eqref{Eq.Direct-sum-decomposition-of-super-lattice}.) We get
		\begin{equation}\label{Eq.Sum-of-Goss-polynomials}
			\sum_{\substack{\widetilde{\Lambda} \in T_{\mathfrak{p},i}(\Lambda) \\ \text{of type 2}}} G_{n}(t_{a}(\widetilde{\Lambda})) = \sum_{\substack{L \text{ as in} \\ \eqref{Eq.Inclusion-of-A-modules}}} \sum_{\mathbf{v} \in V_{L}} \sum_{\mathbf{b} \in L} \frac{1}{(a \frac{\omega_{1}}{\pi} + a \mathbf{v} \boldsymbol{\omega}' + \mathbf{b} \boldsymbol{\omega}')^{n}}.
		\end{equation}
		\item Suppose that \fbox{\(a \in \mathfrak{p}\)}. Then \(a\mathbf{v} \in A^{r-1} \subset L\) for \(\mathbf{v} \in V_{L}\), so the term 
		\(a\mathbf{v} \boldsymbol{\omega}'\) in the denominator may be omitted, the summation over \(\mathbf{v}\in V_{L}\) is multiplication by
		\(\lvert V_{L} \rvert\), and is thus 0. Therefore, \(\sum_{\widetilde{\Lambda} \text{ of type 2}} G_{n}(t_{a}(\widetilde{\Lambda})) = 0\) in
		this case.
		\item Now suppose that \fbox{\(a \notin \mathfrak{p}\)}. Our sum \eqref{Eq.Sum-of-Goss-polynomials} is
		\[
			\sum_{\mathbf{c} \in \pi^{-1}A^{r-1}} \frac{\nu(\mathbf{c})}{(a \frac{\omega_{1}}{\pi} + \mathbf{c} \boldsymbol{\omega}')^{n}},
		\]
		where now \(\nu(\mathbf{c})\) is the number of triples \((L, \mathbf{v} \in V_{L}, \mathbf{b} \in L)\) with \(L\) as in 
		\eqref{Eq.Inclusion-of-A-modules} and \(a\mathbf{v} + \mathbf{b} = \mathbf{c}\).
		
		If \fbox{\(\mathbf{c} \in A^{r-1}\)}, then \(\mathbf{v}\) must vanish and \(\mathbf{b} = \mathbf{c}\), so 
		\(\nu(\mathbf{c}) = \lvert \{ L \} \rvert = c_{r-1,i-1}(\mathfrak{p})\).
		
		If \fbox{\(\mathbf{c} \notin A^{r-1}\)} and \fbox{\(i \geq 2\)}, the triples that solve \(a \mathbf{v}+ \mathbf{b} = \mathbf{c}\) are 
		\((L, \mathbf{0}, \mathbf{c})\) with \(\mathbf{c} \in L\) and \((L, \mathbf{v}, \mathbf{c} - a\mathbf{v})\) with \(\mathbf{c} \notin L\),
		\(a \mathbf{v} \equiv \mathbf{c} \pmod{L}\). 
		
		The former ones are \(c_{r-1,i-1}(\mathfrak{p})\) in number, the latter ones are \(\lvert \{ L \mid \mathbf{c}\notin L\}\rvert = c_{r-1,i-1}(\mathfrak{p}) - c_{r-2,i-2}(\mathfrak{p})\) many.
		
		Finally, if \fbox{\(\mathbf{c} \notin A^{r-1}\)} and \fbox{\(i=1\)}, there is only \(L = A^{r-1}\), and the unique solution triple is 
		\((A^{r-1}, \mathbf{v}, \mathbf{c} - a\mathbf{v})\) with \(a\mathbf{v} \equiv \mathbf{c} \pmod{A^{r-1}}\).
		
		As all the \(c_{*,*}(\mathfrak{p})\) are congruent to \(1 \pmod{p}\), \(\nu(\mathbf{c}) \equiv 1 \pmod{p}\), and the sum 
		\eqref{Eq.Sum-of-Goss-polynomials} becomes
		\[
			\sum_{\mathbf{c} \in \pi^{-1}A^{r-1}} \frac{1}{(a \frac{\omega_{1}}{\pi} + \mathbf{c}\boldsymbol{\omega}')^{n}} = \pi^{n} \sum_{\mathbf{c} \in A^{r-1}} \frac{1}{(a\omega_{1} + \mathbf{c}\boldsymbol{\omega}')^{n}} = \pi^{n} G_{n}(t_{a}(\boldsymbol{\omega})).
		\]
		The result now follows from (ii), (iv) and (v).
	\end{enumerate}
\end{proof}

\begin{Remark}
	Like on the Eisenstein series \(E_{k}\), the different \(T_{\mathfrak{p},i}\) (\(\mathfrak{p} = (\pi)\) fixed) do not differ on the
	\(G_{n}(t_{a})\). In fact, we can easily derive the Hecke action on the \(E_{k}\) from that on the \(G_{k}(t_{a})\). We label the operators by their
	respective ranks \(r\) and \(r-1\) and use induction. Then
	\begin{align*}
		T_{\mathfrak{p},i}^{(r)}(E_{k}^{(r)})	&= T_{\mathfrak{p},i}^{(r)}(E_{k}^{(r-1)}) - \sum_{a \text{ monic}} G_{k}(t_{a}) ) 	\qquad \text{by \eqref{Eq.Ek-omega-as-Ek-omega-prime-minus-Goss-polynomials}	} \\
												&= T_{\mathfrak{p},i}^{(r-1)} E_{k}^{(r-1)} - \pi^{k} \sum_{\substack{a \text{ monic} \\ a \in \mathfrak{p}}} G_{k}(t_{a\pi}) - \pi^{k} \sum_{\substack{a \text{ monic} \\ (a, \mathfrak{p})=1}} (G_{k}(t_{a\pi}) + G_{k}(t_{a}))
		\intertext{(by Propositions \ref{Proposition.T-p-i-on-lattice} and \ref{Proposition.Action-of-T-p-i-on-Goss-polynomials})}
												&= \pi^{k} E_{k}^{(r-1)} - \pi^{k} \sum_{a \text{ monic}} G_{k}(t_{a}) = \pi^{k} E_{k}^{(r)}
	\end{align*}
	(by induction hypothesis: \(E_{k}^{(r-1)}\) is the restriction of \(E_{k}^{(r)}\) to the boundary). We hope that this proof scheme may eventually
	be applied to a larger class of modular forms that have \(A\)-expansions through the \(G_{k}(t_{a})\) in the style of Petrov \cite{Petrov13}.
\end{Remark}

\section{The Hecke action on the basic modular forms}\label{Section.The-Hecke-action-on-the-basic-modular-forms}

We allow \(r \geq 2\) but assume for the largest part of the section that the type \(\ell\) of a modular form \(f\) is 0. Then it is in fact a
power series in \(t^{q-1}\).

Let us consider the following

\begin{Property}\label{Property.Power-series-in-t}
	Let \(f(t) = \sum a_{n}t^{n}\) be a power series in \(t\). The coefficient \(a_{n}\) vanishes identically unless \(n \equiv 0 \pmod{q-1}\)
	and \(n \equiv 0 \text{ or } {-}1 \pmod{q}\).
\end{Property}

It is shared by all the basic modular forms.

\begin{Proposition}\label{Proposition.Forms-gi-satisfy-Property-5-1}
	The forms \(g_{i}\) (\(1 \leq i \leq r\)), the \(\alpha_{i}\) (\(i \in \mathds{N}\)), and the special Eisenstein series \(E_{q^{i}-1}\)
	(\(i \in \mathds{N}\)), regarded as power series in \(t\), satisfy Property 	\ref{Property.Power-series-in-t}, in particular the discriminant
	\(\Delta = g_{r}\).
\end{Proposition}

\begin{proof}[Proof, see {\cite{Gekeler88} 6.10}]
	\begin{enumerate}[wide, label=(\roman*)]
		\item In view of the relations \eqref{Eq.Basic-coefficient-form-Eisenstein-series} and \eqref{Eq.Coefficients-of-exponential-function}
		between the three families of modular forms, which preserve the property, it suffices to treat the \(E_{q^{i}-1}\).
		\item The polynomial \(S_{a}(t)\) of \eqref{Eq.Polynomial-expression-for-Sn} satisfies \ref{Property.Power-series-in-t}.
		\item Let \(k \defeq q^{i}-1\) and \(a \in A\) be monic of degree \(d\). Then
		\[
			t_{a}^{k} = \mathrm{const.} \cdot (t^{kq^{(r-1)d}} S_{a}^{{-}q^{i}}) S_{a}
		\]
		satisfies \ref{Property.Power-series-in-t}, as the first factor is a \(q\)-th power.
		\item By (iii) and the special form of \(G_{q^{i}-1,\Lambda'}(X)\), whose support is \(q^{i}-1\), \(q^{i}-q\), \dots, \(q^{i} - q^{i-1}\)
		by \ref{Proposition.Characterization-Goss-polynomials}(viii), the power series \(G_{q^{i}-1,\Lambda'}(t_{a})\) satisfies
		\ref{Property.Power-series-in-t}.
		\item Hence by the expansion of \eqref{Eq.Ek-omega-as-Ek-omega-prime-minus-Goss-polynomials}, the result follows for \(E_{q^{i}-1}\).
	\end{enumerate}	
\end{proof}

\begin{Remarks}
	\begin{enumerate}[wide, label=(\roman*)]
		\item Property \ref{Property.Power-series-in-t} for \(\Delta\) could also be derived from the product formula 
		\eqref{Eq.Product-expression-discriminant-Delta-omega}, but in a more laborious fashion.
		\item The property also holds for general coefficient forms \(\prescript{}{a}\ell_{i}\) (\(a \in A\) arbitrary). Again this follows from
		the relation (see e.g. \cite{Gekeler88} 2.10) between the \(\prescript{}{a}\ell_{i}\) and the \(E_{q^{i}-1}\) that generalizes 
		\eqref{Eq.Basic-coefficient-form-Eisenstein-series}.
	\end{enumerate}	
\end{Remarks}

\subsection{} Consider the algebra \(\mathbf{M}_{0} = \mathbf{M}_{0}^{r} = \bigoplus_{k \geq 0} M_{k,0}^{r}\) of modular forms of type zero
for \(\Gamma = \GL(r,A)\). In the following, we often change the rank \(r\), and therefore label objects by the rank is necessary.

Restricting a modular form \(f \in M_{k,0}^{r}\) to the boundary \(\partial \Omega^{r} = \bigcupdot_{1 \leq j < r} \Omega^{j}\) of \(\Omega^{r}\),
there results a modular form \(f' \in M_{k,0}^{(r-1)}\) for \(\Gamma' = \GL(r-1,A)\)\footnote{There are no derivatives of modular forms in 
this paper.}. We thus get a map 
\begin{equation}
	\begin{split}
		\res_{r-1}^{r} \colon \mathbf{M}_{0}^{r} = C_{\infty}[g_{1}^{(r)}, \dots, g_{r}^{(r)}]	&\longrightarrow \mathbf{M}_{0}^{r-1} = C_{\infty}[g_{1}^{(r-1)}, \dots, g_{r-1}^{(r-1)}] \\
																					g_{i}^{(r)}	&\longmapsto g_{i}^{(r-1)} \qquad (1 \leq i < r) \\
																	g_{r}^{(r)} = \Delta^{(r)}	&\longmapsto 0	
	\end{split}
\end{equation}
with kernel the ideal of cusp forms. By Proposition \ref{Proposition.T-p-i-on-lattice}, the diagram
\begin{equation}\label{Eq.Isomorphism-of-Hecke-modules}
	\begin{tikzcd}
		\mathbf{M}_{0}^{r} \ar[r, "T_{\mathfrak{p},i}^{(r)}"] \ar[d, "\res_{r-1}^{r}"']	& \mathbf{M}_{0}^{r} \ar[d, "\res_{r-1}^{r}"]\\
		\mathbf{M}_{0}^{r-1} \ar[r, "T_{\mathfrak{p},i}^{(r-1)}"]						& \mathbf{M}_{0}^{r-1}
	\end{tikzcd}
\end{equation}
commutes for each \(i\), \(1 \leq i < r\). Put 
\[
	\mathbf{M}_{<q^{r}-1,0}^{r} \defeq \bigoplus_{k < q^{r}-1} M_{k,0}^{r};
\]
then \(\res_{r-1}^{r}\) defines an isomorphism of Hecke modules
\begin{equation}\label{Eq.Isomorphism-of-Hecke-modules}
	\mathbf{M}_{<q^{r-1},0}^{r} \overset{\cong}{\longrightarrow} \mathbf{M}_{<q^{r}-1,0}^{r-1}.
\end{equation}

\subsection{} Let \(0 \neq f\) be a modular form in (a) \(M_{k,0}\) resp. (b) \(S_{k,0}\), and assume that 
\begin{equation}
	\text{(a)} \dim M_{k,0} = 1 \quad \text{resp.} \quad \text{(b)} \dim S_{k,0} = 1.
\end{equation}
Then we know a priori that \(f\) is an eigenform for all the \(T_{\mathfrak{p},i}\). The condition holds if
\begin{enumerate}
	\item[(a)] \(k = j(q-1)\), \(1 \leq j \leq q\)	
\end{enumerate}
resp.
\begin{enumerate}
	\item[(b)] \(k = q^{r}-1 + j(q-1)\), \(0 \leq j \leq q\); 
\end{enumerate}
respective basis vectors are (a) \(g_{1}^{j}\) resp. (b) \(\Delta g_{1}^{j}\).

\begin{Proposition}\label{Proposition.Eigenforms-of-T-p-i}
	All the forms \(g_{1}^{j}\) (\(1 \leq j \leq q\)), \(g_{2}g_{1}^{j}\), \dots, \(g_{r-1}g_{1}^{j}\), \(\Delta g_{1}^{j}\) \((0 \leq j \leq q)\)
	are eigenforms for all the \(T_{\mathfrak{p},i}\) \((1 \leq i \leq r)\). The eigenvalue of \(T_{\mathfrak{p},k}\) on \(g_{i}g_{1}^{j}\) may
	be calculated in \(\mathbf{M}_{0}^{i}\), that is, as an eigenvalue on \(\Delta^{(i)}(g_{1}^{(i)})^{j}\) \((2 \leq i < r\), \(0 \leq j \leq q\),
	\(1 \leq k \leq i)\).	
\end{Proposition}

\begin{proof}
	We use \eqref{Eq.Isomorphism-of-Hecke-modules} and induction on \(r\).
	
	\fbox{\(r=2\)} The relevant spaces \(M_{k,0}^{(2)}\) and \(S_{k,0}^{(2)}\) are \(1\)-dimensional, and thus give eigenforms for 
	\(T_{\mathfrak{p}} \defeq T_{\mathfrak{p},1}^{(2)}\) and \(T_{\mathfrak{p},2}^{(2)}\).
	
	\fbox{\(r>2\)} Assume the assertion holds for \(r-1\). The stated forms of rank \(r\) and weight \(< q^{r}-1\) (i.e., those involving
	\(g_{i}^{(r)}\) with \(i < r\)) are eigenforms for \(T_{\mathfrak{p},i}^{(r)}\) (\(1 \leq i < r\)) by \eqref{Eq.Isomorphism-of-Hecke-modules} and
	the induction hypothesis, and are a priori eigenforms for \(T_{\mathfrak{p},r}^{(r)}\). The missing ones: \(\Delta^{(r)}(g_{1}^{(r)})^{j}\) lie
	in the \(1\)-dimensional spaces \(S_{q^{r}-1+j(q-1)}^{(r)}\) and are thus eigenforms, too.
\end{proof}

\begin{Remarks}\label{Remarks.Regarding-Eigenforms-of-T-p-i}
	\begin{enumerate}[wide, label=(\roman*)]
		\item The basic coefficient form \(g_{1}\) equals \((T^{q}-T)E_{q-1}\) by \eqref{Eq.Basic-coefficient-form-Eisenstein-series}, so
		\(g_{1}^{j} = (T^{q}-T)^{j} E_{j(q-1)}\) for \(1 \leq j \leq q\) by property \ref{Proposition.Characterization-Goss-polynomials}(iv)
		of Goss polynomials. Hence the eigenvalue of \(T_{\mathfrak{p},i}\) on \(g_{1}^{j}\) is \(\pi^{j(q-1)}\) for \(1 \leq i \leq r\)
		and \(1 \leq j \leq q\).
		\item The eigenvalue of \(T_{\mathfrak{p},i}\) (\(i=1,2\)) on \(g_{2}^{(r)}\) equals the eigenvalue of \(T_{\mathfrak{p},i}\) on
		\(g_{2}^{(2)} = \Delta^{(2)}\), which is \(\pi^{q-1}\) for \(i = 1\) (\cite{Gekeler88} 7.5) and \(\pi^{q^{2}-1}\) for \(i = 2\).
		\item For the moment we allow non-trivial types. For \(0 \leq j < q-1\), the form \(h^{j} = (h^{(r)})^{j}\) is a basis vector for
		\(M_{jw,j}\) (\(w \defeq (q^{r}-1)/(q-1)\)), and therefore an eigenform. For essentially trivial reasons (see Corollary 5.9), we get
		\(\pi^{(q^{i}-1)/(q-1)}\) as the eigenvalue of \(T_{\mathfrak{p},i}\) on \(h\). But already for \(h^{2}\), the corresponding 
		investigations are non-trivial and not yet accomplished. Note that in rank \(r=2\), the \((h^{(2)})^{j}\) have eigenvalue \(\pi^{j}\) under
		\(T_{\mathfrak{p},1}^{(2)}\) by \cite{Petrov13} Theorem 2.3 and 3.17.
	\end{enumerate}	
\end{Remarks}

Now we are in a position to state and prove the main result.

\begin{Theorem}\label{Theorem.Basic-coefficient-forms-are-Eigenforms-of-Hecke-operators}
	\begin{enumerate}[label=\(\mathrm{(\roman*)}\)]
		\item The basic coefficient forms \(g_{1}, g_{2}, \dots, g_{r}= \Delta^{(r)}\) are eigenforms for the Hecke operators 
		\(T_{\mathfrak{p},i}\) \((1 \leq i \leq r)\).
		\item For \(1 \leq i,j \leq r\), let \(\lambda_{i,j}\) be the eigenvalue of \(T_{\mathfrak{p},i}\) on \(g_{j}\). Then 
		\begin{equation}\label{Eq.Characterization-eigenvalues-T-p-i-on-g-j}
			\lambda_{i,j} = \pi^{q^{\min(i,j)} - 1},
		\end{equation}
		where \(\pi\) is the monic generator of the prime ideal \(\mathfrak{p}\).
	\end{enumerate}	
\end{Theorem}

\begin{proof}
	\begin{enumerate}[wide, label=(\roman*)]
		\item comes from \ref{Proposition.Eigenforms-of-T-p-i}.
		\item We use induction on \(r \geq 2\). For \fbox{\(r=2\)}, the values \(\lambda_{1,1}\) and \(\lambda_{1,2}\) are as wanted by 
		Remarks \ref{Remarks.Regarding-Eigenforms-of-T-p-i}(ii). Further, for arbitrary \(r \geq 2\)
		\begin{equation}\label{Eq.Determination-eigenvalue-lambda-r-j}
			\lambda_{r,j} = \pi^{q^{j}-1} \qquad \text{by \eqref{Eq.Association-lattice-finite-collection-T-p-i}}.
		\end{equation} 
		Now let \fbox{\(r > 2\)} and assume that \eqref{Eq.Characterization-eigenvalues-T-p-i-on-g-j} holds for \(r' = r-1\). By 
		\eqref{Eq.Isomorphism-of-Hecke-modules} and \eqref{Eq.Determination-eigenvalue-lambda-r-j}, it remains to determine the \(\lambda_{i,r}\)
		with \(1 \leq i < r\).
		
		By the product formula \eqref{Eq.Product-expression-discriminant-Delta-omega}, the \(t\)-expansion of \(\Delta(\boldsymbol{\omega})\)
		starts 
		\[
			\Delta(\boldsymbol{\omega}) = {-}\Delta'(\boldsymbol{\omega}')^{q}t^{q-1} + \sum_{n > q-1} a_{n}(\boldsymbol{\omega}')t^{n}(\boldsymbol{\omega}).
		\]
		(primed data refer to objects of rank \(r' = r-1\)). Therefore, we must calculate the \(t^{q-1}\)-term of \(T_{\mathfrak{p},i}(\Delta)\).
		Inspection of \eqref{Eq.t-expansion-of-f-in-M-k-l-to-power-series-expansion-of-T-p-i-f} reveals that terms of type 1 cannot contribute
		to the \(t^{q-1}\)-term.
		
		\textbf{Claim:} Let \(n > q-1\) be divisible by \(q-1\). Then (notation as in \eqref{Eq.t-expansion-of-f-in-M-k-l-to-power-series-expansion-of-T-p-i-f}) \(a_{n}(\widetilde{\Lambda}') G_{n}^{\widetilde{\Lambda}'}(X)\) has no \(X^{q-1}\)-term for \(\widetilde{\Lambda}'\) of type 2.
		
		\textit{Proof of the claim} (see \cite{Gekeler88}, proof of Corollary 7.5). We write \(G_{n}\) for the Goss polynomial 
		\(G_{n}^{\widetilde{\Lambda}'}\). Suppose \(a_{n}(\widetilde{\Lambda}') \neq 0\). Then, by Proposition 
		\ref{Proposition.Forms-gi-satisfy-Property-5-1}, \(n \equiv 0 \text{ or } {-}1 \pmod{q}\). If \(n \equiv 0 \pmod{q}\) then \(G_{n}\) is
		a \(q\)-th power (Proposition \ref{Proposition.Characterization-Goss-polynomials}(v)) and has no \(X^{q-1}\)-term. Otherwise, if 
		\(n \equiv {-}1 \pmod{q}\), \(n+1 = mq\) with \(m > 1\), we use Proposition \ref{Proposition.Characterization-Goss-polynomials}(vi) to
		find
		\[
			X^{2} \frac{\mathop{d}}{\mathop{dx}} G_{n}(X) = {-}G_{mq}(X) = {-}G_{m}(X)^{q}.
		\]
		As \(m > 1\), \(G_{m}\) has no \(X\)-term, so \(G_{n}\) has no \(X^{q-1}\)-term, and the claim is verified.
		
		By the claim, only the Goss polynomial \(G_{q-1}^{\widetilde{\Lambda}'}(\pi t)\) can contribute to the \(t^{q-1}\)-term of 
		\(T_{\mathfrak{p},i} \Delta(\boldsymbol{\omega})\). Now \(G_{q-1}^{\widetilde{\Lambda}'}(X) = X^{q-1}\) independently of 
		\(\widetilde{\Lambda}'\) (Proposition \ref{Proposition.Characterization-Goss-polynomials}(iv)). The \(\widetilde{\Lambda}'\) run
		through the super-lattices of \(\Lambda'\) with \(\widetilde{\Lambda}'/\Lambda'\cong \mathds{F}_{\mathfrak{p}}^{i-1}\), i.e., through
		\(T_{\mathfrak{p},i-1}^{(r-1)}(\widetilde{\Lambda}')\). We find for the coefficient of \(t^{q-1}\) in 
		\(T_{\mathfrak{p},i}\Delta(\boldsymbol{\omega})\):
		\begin{align*}
			a_{q-1}(T_{\mathfrak{p},i} \Delta)	&= \sum_{\widetilde{\Lambda}' \in T_{\mathfrak{p},i-1}^{(r-1)}} a_{q-1}(\widetilde{\Lambda}')\pi^{q-1} \\
												&= {-}\sum_{\widetilde{\Lambda}' \in T_{\mathfrak{p},i-1}^{(r-1)}} \Delta'(\widetilde{\Lambda}')^{q} \pi^{q-1} \\
												&= {-}T_{\mathfrak{p},i-1}^{(r-1)}( (\Delta^{(r-1)})^{q}) \pi^{q-1} \\
												&= {-}(T_{\mathfrak{p},i-1}^{(r-1)}(\Delta^{(r-1)}))^{q} \pi^{q-1} &&\text{(by Proposition \ref{Proposition.Identity-for-C-infty-valued-functions-on-L-or-Lpm})}	
		\end{align*}
		In case \(i=1\), the operator \(T_{\mathfrak{p},i-1}^{(r-1)} = T_{\mathfrak{p},0}^{(r-1)}\) is the identity operator. Thus, with
		\(\lambda_{0,j} = 1\), and applying the induction hypothesis,
		\[
			a_{q-1}(T_{\mathfrak{p},i}\Delta) = \pi^{q-1} \lambda_{i-1,r-1}^{q} a_{q-1}(\Delta),
		\]
		which finally gives
		\begin{equation}
			\lambda_{i,r} = \pi^{q-1} \lambda_{i-1,r-1}^{q} \qquad (r>2, 1 \leq i < r).
		\end{equation}
		Then \(\lambda_{i,r} = \pi^{q^{i}-1}\) is immediate.
	\end{enumerate}
\end{proof}

\begin{Corollary}\label{Corollary.Eigenvalue-of-h-is-pi-power}
	The eigenvalue of \(T_{\mathfrak{p},i}\) on the form \(h \in M_{(q^{r}-1)/(q-1),1}\) is \(\pi^{(q^{i}-1)/(q-1)}\).	
\end{Corollary}

\begin{proof}
	We let \(h^{(j)}\) be the corresponding form in \(M_{(q^{j}-1)/(q-1),1}^{j}\), where \(2 \leq j \leq r\), and put \(\mu_{i,j}\) for its eigenvalue
	under \(T_{\mathfrak{p},i}\) (\(0 \leq i \leq j\)). As
	\[
		h = h^{(r)} = (h^{(r-1)})^{q}t + \text{higher terms},
	\]		
	we must determine the first coefficient \(a_{1}(T_{\mathfrak{p},i}h)\) of \(T_{\mathfrak{p},i}h\). As a substitute for the Claim in the proof of
	Theorem \ref{Theorem.Basic-coefficient-forms-are-Eigenforms-of-Hecke-operators}, we use the trivial fact that Goss polynomials \(G_{n}(X)\) with
	\(n > 1\) have no \(X\)-term. Hence the same argument as for \(\Delta\) yields
	\begin{equation}
		\mu_{i,r} = \pi \mu_{i-1,r-1}^{q} \qquad (r > 2, 1 \leq i \leq r)
	\end{equation}
	with the consequence \(\mu_{i,r} = \pi^{(q^{i}-1)/(q-1)}\).
\end{proof}

\begin{Remark}
	What can we say about eigenvalues on powers of \(h\), say, on \(h^{2}\)? If \(q=2\) then the eigenvalues are those on \(h\), squared; if
	\(q=3\) then \(h^{2} = h^{q-1} = \pm \Delta\). Hence we may assume \(q>3\). We must evaluate the second term \(a_{2}(T_{\mathfrak{p},i}(h^{2}))\),
	for which only those \(G_{n}^{\widetilde{\Lambda}'}(X)\) in \eqref{Eq.t-expansion-of-f-in-M-k-l-to-power-series-expansion-of-T-p-i-f} may contribute
	that have a non-trivial \(X^{2}\)-term. Possible such \(n\) satisfy \(n \equiv 2 \pmod{q-1}\) and \(n \leq q^{(r-1)d}\), where 
	\(d = \deg \mathfrak{p}\). So it is difficult to get control of \(a_{2}(T_{\mathfrak{p},i}(h^{2}))\) in this case. However, for \(r=2\) see
	Remarks \ref{Remarks.Regarding-Eigenforms-of-T-p-i}(iii).	
\end{Remark}

\section{Growth of coefficients}\label{Section.Growth-coefficients}

We study in more detail the congruence and growth properties of the product expansion \eqref{Eq.Product-expression-discriminant-Delta-omega} 
of \(\Delta\) and of similar series, and borrow from ideas of \cite{Gekeler99}.

\subsection{} Let
\begin{equation}\label{Eq.Product-function}
	P(t) = \prod_{n \in A \text{ monic}} S_{n}(t) = \sum_{k \geq 0} p_{k}t^{k} = 1 + o(t^{q-1}) 
\end{equation}
be the \textbf{product function}; so
\begin{align}
	\Delta(\boldsymbol{\omega})	&= {-}( \Delta'(\boldsymbol{\omega}') )^{q} t^{q-1}(\boldsymbol{\omega}) P(t(\boldsymbol{\omega}))^{(q^{r}-1)(q-1)} \\
		h(\boldsymbol{\omega})	&= (h'(\boldsymbol{\omega}'))^{q} t(\boldsymbol{\omega}) P(t(\boldsymbol{\omega}))^{q^{r}-1}. \nonumber 
\end{align}
A priori, these are formal series in \(t\) convergent for sufficiently small values of \(t\); we will determine the convergence radii below.
The \(p_{k}\) are, like \(\Delta'\) and \(h'\), functions on \(\Omega' = \Omega^{r-1}\). In order to perform our calculations, we work on the
fundamental domain \(\mathbf{F} = \mathbf{F}^{r} \subset \Omega = \Omega^{r}\) with closure 
\(\mathbf{F} = \mathbf{F}^{r} \cup \mathbf{F}^{r-1} \cup \cdots \cup \mathbf{F}_{1}\), where \(\mathbf{F}' = \mathbf{F}^{r-1}\) is identified
with \(\{ (0,\omega_{2}, \dots, \omega_{r-1},1) \}\) and correspondingly \(\mathbf{F}^{1} = \{ (0,\dots,0,1) \}\).

\subsection{} We first study \(t\) as a function on \(\mathbf{F}\). As it vanishes nowhere, its logarithm 
\(\log t(\boldsymbol{\omega}) \defeq \log_{q} \lvert t(\boldsymbol{\omega}) \rvert\) is constant on fibers of the building map \(\lambda\) 
and interpolates linearly on simplices of \(W(\mathds{Q}) = \lambda(\mathbf{F})\), see 
\ref{Subsubsection.Invertible-function-on-open-affinoid-subspace} and \ref{Subsubsection.Invertible-function-on-pre-image-of-closed-simplex}. 
Therefore, we may for most questions even restrict to the subsets 
\begin{equation}\label{Eq.Subset-F-k-of-F}
	\mathbf{F}_{k} = \lambda^{-1}([L_{\mathbf{k}}])
\end{equation}
of \(\mathbf{F}\) (see \eqref{Eq.Distinguished-subdomain-of-F-defined-by-k}). Here \(\mathbf{k} \in \mathds{N}_{0}^{r}\) is a fundamental index,
which means \(k_{1} \geq k_{2} \geq \cdots \geq k_{r} = 0\). Throughout, \(r \geq 2\) is fixed. As usual, 
\(\boldsymbol{\omega} = (\omega_{1}, \boldsymbol{\omega}')\) with \(\boldsymbol{\omega}' \in \mathbf{F}'\) and \(\mathbf{k} = (k_{1}, \mathbf{k}')\).
For \(\boldsymbol{\omega} \in \mathbf{F}\), 
\[
	t(\boldsymbol{\omega}) = (e^{\Lambda_{\boldsymbol{\omega}'}}(\omega_{1}))^{-1} = \Big[ \omega_{1} \prod_{\mathbf{a} \in A^{r-1}} \Big(1 - \frac{\omega_{1}}{\mathbf{a} \boldsymbol{\omega}'} \Big)\Big]^{-1},
\]
\(\mathbf{a} = (a_{2}, \dots, a_{r})\), \(\mathbf{a} \boldsymbol{\omega}' = \sum_{2 \leq i \leq r} a_{i}\omega_{i}\). 

The absolute value of the factor \(1-\frac{\omega_{1}}{\mathbf{a}\boldsymbol{\omega}'}\) is 
\begin{align}
	\left\lvert 1 - \frac{\omega_{1}}{\mathbf{a} \boldsymbol{\omega}'} \right\rvert 	&= 1 \quad \text{if} \quad \lvert \mathbf{a} \boldsymbol{\omega}' \rvert \geq \lvert \omega_{1} \rvert \label{Eq.Absolute-value-of-1-minus-omega1-over-a-omega-prime} \\
																					&= \left\lvert \frac{\omega_{1}}{\mathbf{a} \boldsymbol{\omega}'} \right\rvert, \quad \text{if} \quad \lvert \mathbf{a}\boldsymbol{\omega}' \rvert \leq \lvert \omega_{1} \rvert \nonumber
\end{align}
(if \(\lvert \mathbf{a}\boldsymbol{\omega}' \rvert = \lvert \omega_{1} \rvert\), we use the orthogonality of entries of \(\boldsymbol{\omega}\)).
Therefore
\begin{align}
	\lvert t(\boldsymbol{\omega}) \rvert 	&= \lvert \omega_{1} \rvert^{-1} \sideset{}{^{\prime}} \prod_{\substack{\mathbf{a} \in A^{r-1} \\ \lvert \mathbf{a} \boldsymbol{\omega}' \rvert \leq \lvert \omega_{1} \rvert }} \left\lvert \frac{\mathbf{a}\boldsymbol{\omega}'}{\omega_{1}} \right\rvert \label{Eq.Absolute-value-of-t-omega}
	\intertext{and}
	\lvert t(\boldsymbol{\omega}) \rvert 	&\leq 1 \text{ on } \mathbf{F}, \qquad \text{where } \lvert t(\boldsymbol{\omega}) \rvert = 1 \Longleftrightarrow \boldsymbol{\omega} \in \mathbf{F}_{\mathbf{0}}, \mathbf{0} = (0,\dots, 0). 
\end{align}

\subsection{} Next we investigate the \textbf{division functions} \(d_{\mathbf{u}}\) for \(\mathbf{u} \in (K/A)^{r}\). Each such \(\mathbf{u}\) is 
represented 
\begin{equation}\label{Eq.Each-such-u-is-represented-with-monic-a}
	\mathbf{u} = (u_{1}, \dots, u_{r}) = n^{-1}(x_{1}, \dots, x_{r}) \quad \text{with monic } n \in A	
\end{equation}
of some degree \(d\), say, and elements \(x_{i}\) of \(A\) of degree \(d_{i} < d\) (\(1 \leq i \leq r\)). Whenever the syntax requires an element
of \(K\) instead of \(K/A\), we insert \(x_{i}/n\) for \(u_{i}\). In particular, \(\lvert u_{i} \rvert \defeq \lvert x_{i}/n \rvert\) and
\(\lvert \mathbf{u}\boldsymbol{\omega} \rvert = \lvert n^{-1} \sum x_{i}\omega_{i} \rvert\).

The division function is defined by
\begin{equation}
	d_{\mathbf{u}}(\boldsymbol{\omega}) = e^{\Lambda_{\boldsymbol{\omega}}}(\mathbf{u}\boldsymbol{\omega}) = \mathbf{u}\boldsymbol{\omega} \sideset{}{^{\prime}} \prod_{\mathbf{a} \in A^{r}} \left(1 - \frac{\mathbf{u}\boldsymbol{\omega}}{\mathbf{a} \boldsymbol{\omega}} \right).
\end{equation}
The \(d_{\mathbf{u}}\) with \(n \mathbf{u} = 0\) are the \(n\)-division points of the Drinfeld module 
\(\phi^{\boldsymbol{\omega}} = \phi^{\Lambda_{\boldsymbol{\omega}}}\), and are meromorphic modular forms of weight \({-}1\) for the congruence
subgroup \(\Gamma(n)\) of \(\Gamma\) (see, e.g., \cite{Gekeler22-2} 2.4). As in \eqref{Eq.Absolute-value-of-1-minus-omega1-over-a-omega-prime},
\begin{align}
	\left\lvert 1 - \frac{\mathbf{u}\boldsymbol{\omega}}{\mathbf{a}\boldsymbol{\omega}} \right\rvert &= 1, \text{ if } \lvert \mathbf{a} \boldsymbol{\omega} \rvert \geq \lvert \mathbf{u} \boldsymbol{\omega} \rvert \\
													&= \left\lvert \frac{\mathbf{u}\boldsymbol{\omega}}{\mathbf{a}\boldsymbol{\omega}} \right\rvert, \text{ if } \lvert \mathbf{a} \boldsymbol{\omega} \rvert \leq \lvert \mathbf{u} \boldsymbol{\omega} \rvert, \nonumber	
\end{align}
and so for \(\boldsymbol{\omega} \in \mathbf{F}\):
\begin{equation}\label{Eq.Product-expansion-d-n-omega}
	\lvert d_{\mathbf{u}}(\boldsymbol{\omega}) \rvert = \lvert \mathbf{u} \boldsymbol{\omega}\rvert \sideset{}{^{\prime}} \prod_{\substack{\mathbf{a} \in A^{r} \\ \lvert \mathbf{a} \boldsymbol{\omega}\rvert \leq \lvert \mathbf{u} \boldsymbol{\omega} \rvert}} \left\lvert \frac{\mathbf{u}\boldsymbol{\omega}}{\mathbf{a}\boldsymbol{\omega}}\right\rvert, 
\end{equation}
a finite product. Now fix \(\mathbf{k}\) as in \eqref{Eq.Subset-F-k-of-F} and assume \(\boldsymbol{\omega} \in \mathbf{F}_{\mathbf{k}}\). We remind
the reader that \(\lvert d_{\mathbf{u}}(\boldsymbol{\omega}) \rvert\) depends only on \(\mathbf{k}\) but not on the choice of 
\(\boldsymbol{\omega} \in \mathbf{F}_{\mathbf{k}}\).

For a denominator \(n \in A\) of degree \(d \geq 1\), we let 
\begin{equation}
	\mathbf{u}_{0}(n) \defeq n^{-1}(T^{d-1},0,\dots,0).
\end{equation}

\begin{Lemma}\label{Lemma.Bound-for-absolute-values}
	The absolute values \(\lvert d_{\mathbf{u}}(\boldsymbol{\omega}) \rvert\) with \(\mathbf{u} \in (K/A)^{r}\) are bounded on \(\mathbf{F}_{\mathbf{k}}\)
	by \(\lvert d_{\mathbf{u}_{0}(n)}(\boldsymbol{\omega}) \rvert\) (n any non-constant monic in \(A\)). The latter all agree with
	\(\lvert d_{\mathbf{u}_{0}}(\boldsymbol{\omega}) \rvert\), where \(\mathbf{u}_{0} = \mathbf{u}_{0}(T) = (T^{-1},0,\dots,0)\).	
\end{Lemma}

\begin{proof}
	For \(\mathbf{u} = (u_{1}, \dots, u_{r})\) as in \eqref{Eq.Each-such-u-is-represented-with-monic-a}, the \(u_{i}\) satisfy 
	\(\lvert u_{i} \rvert \leq q^{-1}\). By \eqref{Eq.Product-expansion-d-n-omega}, \(\lvert d_{\mathbf{u}}(\boldsymbol{\omega}) \rvert\) depends
	only on \(\lvert \mathbf{u}\boldsymbol{\omega} \rvert\) and increases monotonically with \(\lvert \mathbf{u} \boldsymbol{\omega} \rvert\).
	In view of the nature of \(\boldsymbol{\omega}\), \(\lvert \mathbf{u}\boldsymbol{\omega} \rvert\) is maximal for \(\mathbf{u}\) with
	\(\lvert u_{1} \rvert = q^{-1}\), which holds for all the \(\mathbf{u}_{0}(n)\).
\end{proof}

\subsection{} We are interested in the sizes of coefficients of the polynomial
\begin{equation}\label{Eq.Polynomial-S-n-as-product}\stepcounter{subsubsection}%
	S_{n}(X) = \sideset{}{^{\prime}} \prod_{\substack{\mathbf{u} \in (K/A)^{r} \\ n\mathbf{u} = 0}} (1 - d_{\mathbf{u}}(\boldsymbol{\omega})X),
\end{equation}
where \(n\) is a fixed monic element of \(A\) of degree \(d \geq 1\) (and still \(\boldsymbol{\omega} \in \mathbf{F}_{\mathbf{k}}\)). As it may be
written as 
\begin{equation}\stepcounter{subsubsection}%
	S_{n}(X) = \Delta(\boldsymbol{\omega})^{-1} X^{q^{rd}} \phi_{n}^{\boldsymbol{\omega}}(X^{-1})
\end{equation}
with the Drinfeld module \(\phi^{\boldsymbol{\omega}}\), the support of \(S_{n}(X)\) is contained in 
\(\{0, q^{rd} - q^{rd-1}, q^{rd} - q^{rd-2}, \dots, q^{rd} - 1\}\). The maximal \(\lvert d_{\mathbf{u}}(\boldsymbol{\omega}) \rvert\) that appears in
\eqref{Eq.Polynomial-S-n-as-product} is with \(\mathbf{u} = \mathbf{u}_{0}(\mathbf{n})\). Let 
\subsubsection{} \(j(\mathbf{k})\) be the maximal \(j \in \{1,2,\dots,r\}\) such that \(k_{1} = k_{j}\). Then 
\(\mathbf{u} = n^{-1}(x_{1}, \dots, x_{j(\mathbf{k})}, x_{j(\mathbf{k})+1}, \dots, x_{r})\) gives rise to the maximal value 
\(\lvert d_{\mathbf{u}}(\boldsymbol{\omega}) \rvert = \lvert d_{\mathbf{u}_{0}(n)}(\boldsymbol{\omega}) \rvert\) if and only if at least one 
of \(x_{1}, \dots, x_{j(\mathbf{k})}\) has maximal degree \(d-1\). Hence:
\subsubsection{}\label{Subsubsection.Number-of-n-with-maximal-absolute-value-d-n-omega}%
The number of such \(\mathbf{u}\) with maximal \(\lvert d_{\mathbf{u}}(\boldsymbol{\omega}) \rvert\) is 
\[
	(q^{jd} - q^{j(d-1)})q^{(r-j)d} = q^{rd} - q^{rd-j} \quad \text{with} \quad j=j(\mathbf{k}).
\]

\subsection{} As we want to get control of the product function \(P\) of \eqref{Eq.Product-function}, we apply the preceding to the situation of
rank \(r' = r-1\), where \(\mathbf{k}\) is replaced with \(\mathbf{k}' = (k_{2}, \dots, k_{r-1},0)\), \(\boldsymbol{\omega}\) with
\(\boldsymbol{\omega}'\), and now
\begin{equation}
	S_{n}(X) = (\Delta'(\boldsymbol{\omega}'))^{-1} X^{q^{(r-1)d}} \phi_{n}^{\boldsymbol{\omega}'}(X^{-1}) = \sideset{}{^{\prime}} \prod_{\substack{\mathbf{u} \in (K/A)^{r-1} \\ n \mathbf{u} = 0}} (1 - d_{\mathbf{u}}(\boldsymbol{\omega}')X)
\end{equation}
as in \eqref{Eq.Polynomial-expression-for-Sn}.

\begin{Lemma}\label{Lemma.Inequality-for-t-omega}
	For each \(\boldsymbol{\omega} \in \mathbf{F}_{\mathbf{k}}\) and \(\mathbf{u} \in (K/A)^{r-1}\), the inequality
	\begin{equation}
		\lvert t(\boldsymbol{\omega}) d_{\mathbf{u}}(\boldsymbol{\omega}') \rvert \leq q^{-1}
	\end{equation}	
	holds.
\end{Lemma}

\begin{proof}
	This follows from combining the equalities \eqref{Eq.Absolute-value-of-t-omega} and \eqref{Eq.Product-expansion-d-n-omega} and taking into account
	that \(\lvert \omega_{1} \rvert \geq \lvert \omega_{2} \rvert > \lvert \mathbf{u}\boldsymbol{\omega}' \rvert\) for each \(\mathbf{u}\).
\end{proof}

\subsection{} Next we choose an element \(\gamma \in C_{\infty}\) that realizes the spectral norm \(\lVert d_{\mathbf{u}_{0}} \rVert_{\mathbf{k}'}\)
of \(d_{\mathbf{u}_{0}}\) on \(\mathbf{F}_{\mathbf{k'}}'\), that is 
\begin{equation}
	\lvert \gamma \rvert = \lVert d_{\mathbf{u}_{0}} \rVert_{\mathbf{k}'} = \lvert d_{\mathbf{u}_{0}}(\boldsymbol{\omega}') \rvert.
\end{equation} 
Performing the coordinate change 
\begin{align}
	d_{\mathbf{u}}^{*} = \gamma^{-1} d_{\mathbf{u}}, t^{*} = \gamma t, p_{k}^{*} = \gamma^{{-}k}p_{k}, S_{n}^{*}(X) = \prod (1 - d_{\mathbf{u}}^{*} X), \\ P(t) = \prod_{n \in A \text{ monic}} S_{n}(t) = \sum_{k \geq 0} p_{k}t^{k} = \sum_{k \geq 0} p_{k}^{*}(t^{*})^{k}, \nonumber 
\end{align}
the inequality \(\lvert p_{k} \rvert \leq \lvert \gamma \rvert^{k}\) that comes from Lemma \ref{Lemma.Bound-for-absolute-values} and 
\eqref{Eq.Polynomial-S-n-as-product} is transformed to \(\lvert p_{k}^{*} \rvert \leq 1\), where
\begin{equation}
	\lvert p_{k}^{*} \rvert = 1 \Longleftrightarrow \lvert p_{k} \rvert = \lvert \gamma \rvert^{k}.
\end{equation}

\subsection{} Let \(\overline{(~)} \colon O_{C_{\infty}} \to \overline{\mathds{F}}\) be the reduction map, where 
\(\overline{\mathds{F}} = O_{C_{\infty}}/\mathfrak{m}_{C_{\infty}}\) is the algebraic closure of \(\mathds{F} = \mathds{F}_{q}\). As 
\(\overline{\mathds{F}}\) lifts to a subfield of \(C_{\infty}\), each \(z \in O_{C_{\infty}}\) has a unique presentation \(z = \overline{z} + x\)
with \(\overline{z} \in \overline{\mathds{F}}\) and \(x \in \mathfrak{m}_{C_{\infty}}\). By some abuse of notation, we write \(\overline{p}_{k}\)
for \( (\overline{p_{k}^{*}})\) and \(\overline{S}_{n}(X) \in \overline{\mathds{F}}[X]\) for \(\overline{S_{n}^{*}(X)}\). By 
\ref{Subsubsection.Number-of-n-with-maximal-absolute-value-d-n-omega}, the degree of \(\overline{S}_{n}(X)\) is 
\begin{equation}
	\deg \overline{S}_{n}(X) = q^{(r-1)d} - q^{(r-1)d-(j-1)},
\end{equation}
where \(j = j(\mathbf{k}')\), i.e., \(k_{2} = \cdots = k_{j} > k_{j+1}\).

\subsection{}\label{Subsection.Replacing-n-for-other-monic-polynomial}%
Replace the monic \(n\) by some monic \(m\) of the same degree \(d\). It is easily seen that the leading coefficients of the various
\(d_{\mathbf{u}}(\boldsymbol{\omega}')\) remain unchanged, and so \(\overline{S}_{n} = \overline{S}_{m}\). Putting 
\begin{align}\stepcounter{subsubsection}%
	P_{d}(X)			&\defeq \prod_{\substack{n \text{ monic} \\ \text{of degree }d}} S_{n}(X)	&&(d>0, P_{0} = 1),
	\intertext{we find}
	\overline{P}_{d}(X)	&= \prod_{\substack{n \text{ monic}\\ \deg n = d}} \overline{S}_{n}(X) = \overline{S}_{n}(X)^{q^{d}}		&&\text{with one fixed \(n\) of degree \(d\)}. \nonumber 
	\intertext{Hence \(\overline{P}_{d}(X)\) has shape}
	\overline{P}_{d}(X)	&= 1 + *X^{q^{rd}-q^{rd-1}} + \cdots + * X^{q^{rd}-q^{rd-(j-1)}}	 \nonumber
\end{align}
with coefficients \(* \in \overline{\mathds{F}}\) depending on \(\boldsymbol{\omega}'\), the last one non-zero. Finally,
\begin{equation}\stepcounter{subsubsection}%
	\overline{P}(X) = \prod_{d \geq 1} \overline{P}_{d}(X) = \prod_{d \geq 1}(1 + *X^{q^{rd}-q^{rd-1}} + \cdots + *X^{q^{rd}-q^{rd-(j-1)}}).
\end{equation}
Its \(k\)-th coefficient \(\overline{p}_{k}\) is non-zero at least if \(k\) is a sum of different numbers of the form
\begin{equation}\stepcounter{subsubsection}%
	Q_{d} \defeq q^{rd} - q^{rd-(j-1)} \qquad (j = j(\mathbf{k}'), d \in \mathds{N} \text{ variable}).
\end{equation}
This results from the following easily checked \enquote{unmixedness} property, which rules out cancellations.
\subsubsection{} Any representation of \(k \in \mathds{N}\) as a sum 
\[
	k = R_{d_{1}} + R_{d_{2}} + \cdots+ R_{d_{s}}
\]
with \(d_{1} > d_{2} > \cdots > d_{s}\) and \(R_{d_{t}} \in \{ q^{rd_{t}} - q^{rd_{t}-1}, \dots, q^{rd_{t}} - q^{rd_{t} - (j-1)}\}\) for 
\(1 \leq t \leq s\) is unique, in the sense that \(s\), the \(d_{t}\) and the \(R_{d_{t}}\) are uniquely determined by \(k\).

We thus get the following result.

\begin{Theorem}\label{Theorem.Fundamental-index}
	Let \(\mathbf{k} = (k_{1}, \dots, k_{r}) \in \mathds{N}_{0}^{r}\) be a fundamental index, that is \(k_{1} \geq k_{2} \geq \cdots \geq k_{r} = 0\),
	\(j = j(\mathbf{k}')\) such that \(k_{2} = \cdots = k_{j} > k_{j+1}\) and \(\mathbf{F}_{\mathbf{k}}\) (resp. \(\mathbf{F}'_{\mathbf{k}'}\))
	the corresponding subdomain of \(\mathbf{F}\) (resp. \(\mathbf{F}'\)). Let further \(c(\mathbf{k}')\) be the spectral norm of \(d_{\mathbf{u}_{0}}\)
	on \(\mathbf{F}_{\mathbf{k}'}\) (i.e., \(c(\mathbf{k}') = \lVert d_{\mathbf{u}_{0}} \rVert_{\mathbf{k}'} = \lvert d_{(T^{-1},0,\dots,0)}(\boldsymbol{\omega}') \rvert\)).
	\begin{enumerate}[label=\(\mathrm{(\roman*)}\)]
		\item As functions on \(\mathbf{F}'_{\mathbf{k}'}\), the coefficients \(p_{k}\) of the product function \(P(t) = \sum_{k \geq 0} p_{k}t^{k}\)
		satisfy
		\begin{equation}\label{Eq.Absolute-value-pk}
			\lvert p_{k} \rvert \leq c(\mathbf{k}')^{k},
		\end{equation}
		with equality at least if \(k\) is a sum of different numbers of the form \(Q_{d} = q^{rd} - q^{rd-(j-1)}\).
		\item The coefficients \(a_{k}\) of 
		\[
			\left(\frac{{-}1}{\Delta^{\prime q} t^{q-1}}\right)\Delta = P(t)^{(q^{r}-1)(q-1)}
		\] 
		and \(b_{k}\) of
		\[ 
			\left(\frac{1}{h^{\prime q}t}\right)h = P(t)^{q^{r}-1}
		\] 
		satisfy the same estimates 
		\begin{equation}\label{Eq.Bound-for-absolute-value-of-an}
			\lvert a_{k} \rvert \leq c(\mathbf{k}')^{k}
		\end{equation}
		and
		\begin{equation}\label{Eq.Bound-for-absolute-value-of-bn}
			\lvert b_{k} \rvert \leq c(\mathbf{k}')^{k}
		\end{equation}
		on \(\mathbf{F}'_{\mathbf{k}'}\).
	\end{enumerate}	
\end{Theorem}

\begin{proof}
	(i) has been shown, as \(\lvert p_{k} \rvert \leq c(\mathbf{k}')^{k}\) is equivalent with \(\lvert p_{k}^{*} \rvert \leq 1\). The 
	inequalities in (ii) are then immediate.
\end{proof}

\begin{Corollary}\label{Corollary.Uniform-convergence-of-infinite-product-P-on-F}
	The infinite product (resp. sum) for \(P\) in \eqref{Eq.Product-function} converges uniformly on \(\mathbf{F}\), and the same is true
	for the corresponding expansions of \(h\) and of \(\Delta\).	
\end{Corollary}

\begin{proof}
	It suffices to consider the case of the product function \(P\). For fixed \(\mathbf{k}'\), the radius of convergence of \eqref{Eq.Product-function}
	is \(c(\mathbf{k}')^{-1}\) by \eqref{Eq.Absolute-value-pk} and the fact that equality holds for infinitely many \(k\). By Lemma 
	\ref{Lemma.Inequality-for-t-omega}, for each fundamental index \(\mathbf{k} = (k_{1}, \mathbf{k}')\) and 
	\(\boldsymbol{\omega} \in \mathbf{F}_{\mathbf{k}}\), \(\lvert t(\boldsymbol{\omega}) \rvert \leq q^{-1}c(\mathbf{k}')^{-1}\), and so 
	\(\mathbf{F}_{\mathbf{k}}\) belongs to the domain of convergence.
	
	Hence \eqref{Eq.Product-function} converges uniformly on \(\mathbf{F}_{\mathbf{k}}\). The uniform convergence on \(\mathbf{F}\) follows, as
	both the \(q\)-logarithms \(\log t(\boldsymbol{\omega}) = \log_{q} \lvert t(\boldsymbol{\omega}) \rvert\) and \(\log_{q} c(\mathbf{k}')\)
	interpolate linearly on simplices of the simplicial complex \(\mathcal{W} = \mathcal{W}^{r}\) (resp. \(\mathcal{W}' = \mathcal{W}^{r-1}\)).
	
	(Recall that \(\mathcal{W} \subset \mathcal{BT}\) is the subcomplex with \(\mathbf{F} = \lambda^{-1}(\mathcal{W})\).)
\end{proof}

\begin{Remarks}
	\begin{enumerate}[wide, label=(\roman*)]
		\item Arguing as in \ref{Subsection.Replacing-n-for-other-monic-polynomial}, but investing a bit more labor, we could figure out 
		infinitely many \(k\) such that equality holds in \eqref{Eq.Bound-for-absolute-value-of-an} (resp. in 
		\eqref{Eq.Bound-for-absolute-value-of-bn}).
		\item Thanks to the Cauchy integral formula, a complex power series has always the largest possible radius of convergence. In our
		situation, the discriminant function (or its root \(h\)) is defined and holomorphic for arbitrary \(\boldsymbol{\omega} \in \Omega\), in
		particular for \(\boldsymbol{\omega}\) with large \(\lvert t(\boldsymbol{\omega}) \rvert\), for which the convergence of 
		\eqref{Eq.Product-function} (or of its \((q^{r}-1)(q-1)\)-th power) fails. This marks an important difference of complex and non-archimedean
		analysis, as there is no substitute of Cauchy's formula in the latter. Fortunately, by Corollary 
		\ref{Corollary.Uniform-convergence-of-infinite-product-P-on-F}, at least the fundamental domain \(\mathbf{F}\) is covered by the domain of
		convergence of the expansions of \(\Delta\) and \(h\).
	\end{enumerate}	
\end{Remarks}

\subsection{} We conclude with analogous but more simple observations on the convergence of the formula 
\eqref{Eq.Ek-omega-as-Ek-omega-prime-minus-Goss-polynomials} for Eisenstein series. It is (with our usual notations; as before, we assume that
\(\boldsymbol{\omega} \in \mathbf{F}_{\mathbf{k}}\)):
\begin{equation}\label{Eq.E-k-omega-is-E-k-omega-prime-minus-Goss-polynomials}
	E_{k}(\boldsymbol{\omega}) = E_{k}(\boldsymbol{\omega}') - \sum_{a \in A \text{ monic}} G_{k, \Lambda'}(t_{a}(\boldsymbol{\omega})),
\end{equation}
where 
\[
	t_{a}(\boldsymbol{\omega}) = t(a\omega_{1}, \boldsymbol{\omega}') = (\Delta_{a}'(\boldsymbol{\omega}'))^{-1} t^{q^{(r-1)\deg a}}/S_{a}(t(\boldsymbol{\omega})).
\]
As on \(\mathbf{F}_{\mathbf{k}}\), \(\lvert t(\boldsymbol{\omega}) \rvert \leq q^{-1} c(\mathbf{k}')^{-1}\), the quantity 
\(S_{a}(t(\boldsymbol{\omega}))\) is non-zero and the ingredients of \eqref{Eq.E-k-omega-is-E-k-omega-prime-minus-Goss-polynomials} 
are well-defined on \(\mathbf{F}_{\mathbf{k}}\). For its convergence, we could elaborate estimates on the coefficients of \(G_{k,\Lambda'}\); 
it is however easier to refer to the derivation of \eqref{Eq.Ek-omega-as-Ek-omega-prime-minus-Goss-polynomials}, where the term 
\(G_{k,\Lambda'}(t_{a}(\boldsymbol{\omega}))\) was identified as
\[
	G_{k, \Lambda'}(t_{a}(\boldsymbol{\omega})) = \sum_{\mathbf{b} \in A^{r-1}} \frac{1}{(a \omega_{1} + \mathbf{b}\boldsymbol{\omega}')^{k}}.
\] 
As \(\lvert a\omega_{1} + \mathbf{b}\boldsymbol{\omega}' \rvert \geq \lvert a\omega_{1}\rvert\), we get
\[
	\lvert G_{k,\Lambda'}(t_{a}(\boldsymbol{\omega})) \rvert \leq \lvert a \rvert^{-k} \lvert \omega_{1} \rvert^{-k} = q^{{-}(d+k_{1})k} \qquad (d \defeq \deg a)
\]
on \(\mathbf{F}_{\mathbf{k}}\). This guarantees the uniform convergence of \eqref{Eq.E-k-omega-is-E-k-omega-prime-minus-Goss-polynomials} on
\(\mathbf{F}_{\mathbf{k}}\), as well as the uniform convergence on \(\mathbf{F}\).

\begin{Corollary}\label{Corollary.Convergence-of-all-t-expansions-of-modular-forms-on-fundamental-domain}
	The \(t\)-expansions of all modular forms \(f \in \mathbf{M}^{r}\) converge uniformly on the fundamental domain \(\mathbf{F}\).	
\end{Corollary}

\begin{proof}
	The \(E_{q^{i}-1}\) with \(1 \leq i \leq r\) generate \(\mathbf{M}_{0}^{r}\), and \(\mathbf{M}^{r} = \mathbf{M}_{0}^{r}[h]\).	
\end{proof}

\end{document}